\documentclass[amstex,12pt,thmsa,sw20lart]{article}
\usepackage{amsfonts, amsmath,amsthm,amssymb,color}

\input amssym.def
\input amssym.tex
\newtheorem{lemma}{Lemma} 

\newtheorem{theorem}{Theorem}
\newtheorem{corollary}{Corollary}
\newtheorem{remark}{Remark}
\newtheorem{definition}{Definition}
\newtheorem{example}{Example}
\newtheorem{problem}{Problem}

\begin{document}

\date{}
\title{Symmetrization, factorization and arithmetic of quasi-Banach function
spaces}
\author{Pawe{\l } Kolwicz\thanks{%
Research partially supported by Ministry of Science and Higher Education of
Poland, Grant number 04/43/DSPB/0094}\thinspace , \ Karol Le\'{s}nik$^{\ast
} $ and Lech Maligranda}
\date{}
\maketitle

\begin{abstract}
\noindent We investigate relations between symmetrizations of quasi-Banach
function spaces and constructions such as Calder\'{o}n--Lozanovski{\u{\i}}
spaces, pointwise product spaces and pointwise multipliers. We show that
under reasonable assumptions the symmetrization commutes with these
operations. We determine also the spaces of pointwise multipliers between
Lorentz spaces and Ces\`{a}ro spaces. Developed methods may be regarded as
an arithmetic of quasi-Banach function spaces and proofs of Theorems \ref%
{multip-komut}, \ref{Lorentz-multi} and \ref{mult-cesaro} give a kind of
tutorial for these methods. Finally, the above results will be used in
proofs of some factorization results.
\end{abstract}

%%%%%%%%%%%%%%%%%%%%

%%%%%%%%%%%%%%%%%%%%%%
\renewcommand{\thefootnote}{\fnsymbol{footnote}}

\footnotetext{%
} 

%%%%%%%%%%%%%%%%%%%%%%%%%%%%%%%%%%%%%%%%%%%%

\section{ Introduction and preliminaries}

The functional $x\mapsto \Vert x\Vert $ on a given vector space $X$ is
called a \textit{quasi-norm} if the following three conditions are
satisfied: $\Vert x\Vert =0$ iff $x=0$; $\Vert ax\Vert =|a|\,\Vert x\Vert
,x\in X,a\in \mathbb{R}$; %{\mathbb R}% 
there exists $C=C_{X}\geq 1$ such that $\Vert x+y\Vert \leq C(\Vert x\Vert
+\Vert y\Vert )$ for all $x,y\in X$. We call $\Vert \cdot \Vert $ a \textit{$%
p$-norm} where $0<p\leq 1$ if, in addition, it is $p$-subadditive, that is, $%
\Vert x+y\Vert ^{p}\leq \Vert x\Vert ^{p}+\Vert y\Vert ^{p}$ for all $x,y\in
X$.

A very important result here is the \textit{Aoki--Rolewicz theorem} (cf. 
\cite[Theorem 1.3 on p. 7]{KPR84}, \cite[p. 86]{Ma04}, \cite[pp. 6--8]{Ma08}%
): if $0<p\leq 1$ is given by $C=2^{1/p-1}$, then there exists an equivalent 
$p$-norm $\Vert \cdot \Vert _{1}$ so that 
\begin{equation*}
\Vert x+y\Vert _{1}^{p}\leq \Vert x\Vert _{1}^{p}+\Vert y\Vert _{1}^{p}~~%
\mathrm{and}~~\Vert x\Vert _{1}\leq \Vert x\Vert \leq 2C\Vert x\Vert _{1}
\end{equation*}%
for all $x,y\in X$. Precisely, 
\begin{equation*}
\Vert x\Vert _{1}=\inf \{(\sum_{k=1}^{n}\Vert x_{k}\Vert ^{p})^{1/p}\colon
x=\sum_{k=1}^{n}x_{k},x_{1},x_{2},\ldots ,x_{n}\in X,n=1,2,\ldots \}
\end{equation*}%
defines such a $p$-norm on $X$. The quasi-norm $\Vert \cdot \Vert $ induces
a metric topology on $X$: in fact a metric can be defined by $d(x,y)=\Vert
x-y\Vert _{1}^{p}$, when the quasi-norm $\Vert \cdot \Vert _{1}$ is $p$%
-subadditive. We say that $X=(X,\Vert \cdot \Vert )$ is a \textit{%
quasi-Banach space} if it is complete for this metric.

A quasi-normed or normed space $E=\left( E,\Vert \cdot \Vert _{E}\right) $
is said to be a \textit{quasi-normed ideal (function) space} or \textit{%
normed ideal (function) space} on $I$, where $I=(0,1)$ or $I=(0,\infty )$
with the Lebesgue measure $m$, if $E$ is a linear subspace of $L^{0}(I)$ and
satisfies the so-called ideal property, which means that if $y\in E,x\in
L^{0}$ and $|x(t)|\leq |y(t)|$ for almost all $t\in I$, then $x\in E$ and $%
\Vert x\Vert _{E}\leq \Vert y\Vert _{E}$. If, in addition, $E$ is a complete
space, then we say that $E$ is a \textit{quasi-Banach ideal space} or a 
\textit{Banach ideal space (a quasi-Banach function space or a Banach
function space)}, respectively. We assume that $E$ has a \textit{weak unit},
i.e., it has a function $x$ in $E$ which is positive a.e. on $I$ (see \cite%
{KA77} and \cite{Ma89}).

A quasi-normed ideal space $\left( E,\Vert \cdot \Vert _{E}\right) $ is
called \textit{normable} if there exists on $E$ a norm $\Vert \cdot \Vert
_{1}$ equivalent to $\Vert \cdot \Vert _{E},$ that is there are constants $%
A,B>0$ such that $A\Vert x\Vert _{1}\leq \Vert x\Vert _{E}\leq B\Vert x\Vert
_{1}$ for all $x\in E.$

Recall that a quasi-normed ideal space $E$ has the \textit{Fatou property}
if $0 \leq x_n \uparrow x \in L^0$ with $x_n \in E$ and $\sup_{n \in \mathbb{%
N}} \|x_n\|_E < \infty$ %\mathbb{N}%
imply that $x \in E$ and $\|x_n\|_E \uparrow \|x\|_E$. Recall also that $E$
is \textit{order continuous} if for every $x \in E$ and any $x_n \rightarrow
0$ a.e. with $0 \leq x_n \leq |x|$ we have $\| x_n \|_E \rightarrow 0$.

The \textit{K\"{o}the dual} (or \textit{associated space}) $E^{\prime }$ to
a quasi-normed ideal space $E$ on $I$ is the space of all $x\in L^{0}(I)$
such that 
\begin{equation}
\Vert x\Vert _{E^{\prime }}=\sup \left\{ \int_{I}|x(t)y(t)|\,dt\colon \Vert
y\Vert _{E}\leq 1\right\} <\infty .  \label{dual}
\end{equation}%
It may happen that $E^{\prime }=\{0\}$ but if $E^{\prime }\neq \{0\}$ (for
example, when $E$ is a Banach ideal space), then $(E^{\prime },\Vert \cdot
\Vert _{E^{\prime }})$ is a Banach ideal space. Observe that $E^{\prime }$
has the Fatou property and if $E$ is a Banach ideal space, then $E$ has the
Fatou property if and only if $E^{\prime \prime }\equiv E$ (cf. \cite[p. 30]%
{LT79} and \cite{Za83}). For $0<p\leq \infty $ we define the \textit{%
conjugate number $p^{\prime }$} by 
\begin{equation}
~~p^{\prime }:=%
\begin{cases}
1 & ~\mathrm{if}~p=\infty , \\ 
p/(p-1) & ~\mathrm{if}~1<p<\infty , \\ 
\infty , & ~\mathrm{if}~0<p\leq 1. \\ 
\end{cases}
\label{prime}
\end{equation}

The \textit{weighted quasi-normed ideal space} $E(w)$, where $w\colon
I\rightarrow (0,\infty )$ is a measurable function (\textit{weight on $I$}),
is defined by the norm $\Vert x\Vert _{E(w)}=\Vert xw\Vert _{E}$.

By a \textit{symmetric space} on $I$ we mean a (quasi-)normed ideal space $%
E=(E,\Vert \cdot \Vert _{E})$ with the additional property that for any two
equimeasurable functions $x\sim y,x,y\in L^{0}(I)$ (that is, they have the
same distribution functions $d_{x}=d_{y}$, where $d_{x}(\lambda )=m(\{t\in
I\colon |x(t)|>\lambda \}),\lambda \geq 0$) and $x\in E$ we have that $y\in
E $ and $\Vert x\Vert _{E}=\Vert y\Vert _{E}$. In particular, $\Vert x\Vert
_{E}=\Vert x^{\ast }\Vert _{E}$, where $x^{\ast }(t)=\mathrm{inf}\{\lambda
>0\colon \ d_{x}(\lambda )\leq t\},\ t\geq 0$.

A symmetric space $E$ has the \textit{majorant property} if $y\in
E,\int_{0}^{t}x^{\ast }(s)\,ds\leq \int_{0}^{t}y^{\ast }(s)\,ds$ for all $%
t\in I$, then $x\in E$ and $\Vert x\Vert _{E}\leq \Vert y\Vert _{E}$. For
example, a symmetric normed space $E$ with the Fatou property or being order
continuous has the majorant property (cf. \cite[p. 105]{KPS82}).

The \textit{dilation operator} $D_{s},s>0,$ is defined by $D_{s}x(t)=x(t/s)$
for $t\in I=\left( 0,\infty \right) $ and 
\begin{equation*}
D_{s}x(t)=\left\{ 
\begin{array}{ccc}
x(t/s) & \text{if} & t<\min \left\{ 1,s\right\} , \\ 
0 & \text{if} & s\leq t<1,%
\end{array}%
\right.
\end{equation*}%
for $t\in I=\left( 0,1\right) $. This operator is bounded in any symmetric
quasi-normed space $E$ on $\mathrm{I}$ (and $\Vert D_{s}\Vert _{E\rightarrow
E}\leq \max (1,s)$ for symmetric normed spaces, see \cite[pp. 96--98]{KPS82}
for $I=(0,\infty )$ and \cite[p. 130]{LT79} for both cases) and in some
nonsymmetric quasi-normed function spaces.

For two ideal (quasi-) normed spaces on $I$ the symbol $E\overset{C}{%
\hookrightarrow }F$ means that the inclusion $E\subset F$ is continuous with
a norm which is not bigger than C, i.e., $\| x\| _{F} \leq C \| x\|_{E}$ for
all $x \in E$. In the case when the embedding $E\overset{C}{\hookrightarrow }%
F$ holds with some (unknown) constant $C>0$ we simply write $%
E\hookrightarrow F$. Moreover, $E = F$ (and $E\equiv F$) means that the
spaces are the same and the norms are equivalent (equal).

More information about normed or Banach ideal spaces and symmetric spaces
can be found, for example, in the books \cite{BS88}, \cite{KA77}, \cite%
{KPS82} and \cite{LT79}. Moreover, information on quasi-normed spaces,
quasi-normed function spaces and symmetric spaces we can find, for example,
in the books \cite{KPR84}, \cite{ORS08} and the papers \cite{Ka84}, \cite%
{KR09}, \cite{Ma04}, \cite{Na17}. %\mathcal{P} %

By $\mathcal{P}$ we denote the set of concave nondecreasing functions $\rho
_{0}\colon \lbrack 0,\infty )\rightarrow \lbrack 0,\infty )$ which are 0
only at 0 and %\mathcal{P}%
we identify $\mathcal{P}$ with set of functions $\rho \colon \lbrack
0,\infty )\times \lbrack 0,\infty )\rightarrow \lbrack 0,\infty )$ by
putting $\rho (s,t)=s\rho _{0}(t/s)$ for $s>0$ and 0 for $s=0$.

%\mathcal{P}%
For two normed ideal spaces $E,F$ on $I$ and $\rho \in \mathcal{P}$ the 
\textit{Calder\'{o}n-Lozanovski{\u{\i}} space} (\textit{construction}) $\rho
(E,F)$ is defined as the set of all $x\in L^{0}(I)$ such that for some $%
x_{0}\in E,x_{1}\in F$ with $\Vert x_{0}\Vert _{E}\leq 1,\Vert x_{1}\Vert
_{F}\leq 1$ and for some $\lambda >0$ we have $|x|\leq \lambda \,\rho
(|x_{0}|,|x_{1}|)$ a.e. on $I$. The norm $\Vert x\Vert _{\rho }=\Vert x\Vert
_{\rho (E,F)}$ of an element $x\in \rho (E,F)$ is defined as the infimum
values of $\lambda $ for which the above inequality holds. It can be shown
that 
\begin{equation*}
\rho (E,F)=\left\{ x\in L^{0}(I)\colon |x|\leq \rho (|x_{0}|,|x_{1}|)~%
\mathrm{a.e.~on}~I\mathrm{~for~some}~x_{0}\in E,x_{1}\in F\right\}
\end{equation*}%
and 
\begin{equation*}
\Vert x\Vert _{\rho \left( E,F\right) }=\inf \left\{ \max \left\{ \Vert
x_{0}\Vert _{E},\Vert x_{1}\Vert _{F}\right\} \colon |x|\leq \rho
(|x_{0}|,|x_{1}|)~\mathrm{a.e.~on}~I\text{,}\ x_{0}\in E,x_{1}\in F\right\} 
\text{.}
\end{equation*}%
If $\rho (u,v)=u^{\theta }\,v^{1-\theta }$ with $0<\theta <1$ we write $%
E^{\theta }F^{1-\theta }$ instead of $\rho (E,F)$ and these are Calder\'{o}n
spaces (Calder\'{o}n product) defined already in 1964 in \cite{Ca64}.
Another important situation, investigated by Calder\'{o}n and independently
by Lozanovski{\u{\i}} in 1964, appears when we put $F\equiv L^{\infty }$
(see \cite{Ca64}, \cite{Lo73}, \cite{Lo78}). We can see that they are
generalizations of Orlicz spaces. Moreover, the \textit{$p$-convexification} 
$E^{(p)}$ of $E$, for $1<p<\infty $, is a special case of Calder\'{o}n
product 
\begin{equation*}
E^{1/p}(L^{\infty })^{1-1/p}=E^{(p)}=\{x\in L^{0}\colon |x|^{p}\in E\}~~%
\mathrm{and}~~\Vert x\Vert _{E^{(p)}}=\Vert |x|^{p}\Vert _{E}^{1/p}.
\end{equation*}%
More information on the Calder\'{o}n--Lozanovski{\u{\i}} spaces can be found
in the books \cite{KPS82}, \cite{Ma89}.

For two quasi-normed ideal spaces $E,F$ on $I$ we can define similarly the
Calder\'{o}n--Lozanovski{\u{\i}} space (construction) $\rho (E,F)$ obtaining
the quasi-normed ideal space (cf. \cite{KMP03} and \cite{Ni85}). Also the
definition of $p$-convexification makes sense even for $0<p<\infty $. Note
also that $E^{(p)}$ may be just a quasi-normed ideal space for $0<p<1$ even
if $E$ is a normed ideal space.

Consider the Hardy operator $H$ and its formal K\"{o}the dual $H^{\ast }$
defined for $x\in L^{0}(I)$ by 
\begin{equation*}
Hx(t)=\frac{1}{t}\int_{0}^{t}x(s)\,ds,~H^{\ast }x(t)=\int_{t}^{l}\frac{x(s)}{%
s}\,ds~\mathrm{with}~l=m(I),~t\in I.
\end{equation*}%
Note that if $0<p<1$, then neither $H$ nor $H^{\ast }$ are bounded on $%
L^{p}(w)$ spaces for any weight $w$ (cf. \cite[p. 41]{KMP07}), therefore we
need to consider their \textquotedblleft r-convexifications" for $0<r<\infty 
$, which are defined by 
\begin{equation*}
H_{r}x=[H(|x|^{r})]^{1/r}~~\mathrm{and}~~H_{r}^{\ast }x=[H^{\ast
}(|x|^{r})]^{1/r},
\end{equation*}%
provided the corresponding integrals are finite. These operators are not
linear but they are $c$-sublinear, that is, 
\begin{equation*}
H_{r}(\lambda x)=|\lambda |H_{r}x~~\mathrm{and}~~H_{r}(x+y)\leq
c\,(H_{r}x+H_{r}y)
\end{equation*}%
and similarly for the operator $H_{r}^{\ast }$, where $c=\max (1,2^{1/r-1})$.

In the case $w(t)=t^{\alpha },\alpha \in \mathbb{R}$ %\mathbb%
and $0<p\leq \infty $ it is easy to prove that if $r\leq p$ and $r(\alpha
+1/p)<1$, then $H_{r}$ is bounded on $L^{p}(w)$ with the norm $\leq
(1-\alpha r-r/p)^{-1/r}$ (see \cite[Theorem 2(i)]{KMP07}). Also if $r\leq p$
and $\alpha +1/p>0$, then $H_{r}^{\ast }$ is bounded on $L^{p}(w)$ with the
norm $\leq (\alpha r+r/p)^{-1/r}$.

Using the Fubini theorem we obtain the following equality 
\begin{equation}
H_{r}H_{r}^{\ast }x(t)=\left[ H_{r}x(t)^{r}+H_{r}^{\ast }x(t)^{r}\right]
^{1/r},~~\mathrm{for}~~t\in I.  \label{equalityHardy}
\end{equation}%
In fact, 
\begin{eqnarray*}
H_{r}H_{r}^{\ast }x(t)^{r} &=&\frac{1}{t}\int_{0}^{t}H_{r}^{\ast
}x(s)^{r}\,ds=\frac{1}{t}\int_{0}^{t}(\int_{s}^{l}\frac{|x(u)|^{r}}{u}%
\,du)\,ds \\
&=&\frac{1}{t}\int_{0}^{t}(\int_{0}^{u}ds)\frac{|x(u)|^{r}}{u}\,du+\frac{1}{t%
}\int_{t}^{l}(\int_{0}^{t}ds)\frac{|x(u)|^{r}}{u}\,du \\
&=&\frac{1}{t}\int_{0}^{t}|x(u)|^{r}du+\int_{t}^{l}\frac{|x(u)|^{r}}{u}%
\,du=H_{r}x(t)^{r}+H_{r}^{\ast }x(t)^{r}.
\end{eqnarray*}%
For two quasi-normed ideal spaces $E,F$ on $I$ the \textit{product space} $%
E\odot F$ is 
\begin{equation*}
E\odot F=\left\{ u\in L^{0}(I)\colon u=x\cdot y~~\mathrm{for~some}~x\in E~%
\mathrm{and}~y\in F\right\} ,
\end{equation*}%
and for $u\in E\odot F$ we put 
\begin{equation*}
\Vert u\Vert _{E\odot F}=\inf \{\Vert x\Vert _{E}\Vert y\Vert _{F}\colon
u=x\cdot y,x\in E,y\in F\}.
\end{equation*}%
First note that the product $E\odot F$ is a linear space thanks to the ideal
property of $E$ and $F$ (see \cite{KLM14}). The space $(E\odot F,\Vert \cdot
\Vert _{E\odot F})$ is a quasi-normed ideal space on $I$ (even if $E,F$ are
normed spaces). More about product spaces with some computations can be
found in \cite{CS14} and \cite{KLM14} ({see also \cite{Bu87,BG87} for the
case of sequence spaces)}.

The space of (pointwise) \textit{multipliers} $M(E,F)$ is defined as 
\begin{equation*}
M\left( E,F\right) =\left\{ x\in L^{0}\colon xy\in F\text{ for each }y\in
E\right\}
\end{equation*}%
with the operator norm 
\begin{equation*}
\left\Vert x\right\Vert _{M\left( E,F\right) }=\sup_{\left\Vert y\right\Vert
_{E}=1}\left\Vert xy\right\Vert _{F}.
\end{equation*}%
Properties and several examples of above constructions are presented in \cite%
{KLM13}, \cite{KLM14}, \cite{MP89}, \cite{Na17}.

We collect below several simple and useful facts.

\begin{remark}
\label{uwagi-rozne} Let $E$ be a quasi-Banach ideal space.

\begin{itemize}
\item[$(i)$] If $D_{2}$ is bounded on $E$, then $D_{2}$ is bounded on $%
E^{(p)}$ for each $p>0$ and $\| D_2\|_{E^{(p)} \rightarrow E^{(p)}} \leq \|
D_2\|_{E \rightarrow E}^{1/p}$.

\item[$(ii)$] If $D_{2}$ is bounded on $E,F$, then $D_{2}$ is bounded on $%
E\odot F$ and $\| D_2\|_{E\odot F \rightarrow E\odot F} \leq \| D_2\|_{E
\rightarrow E} \| D_2\|_{F \rightarrow F}$.

\item[$(iii)$] If $H$ is bounded on $E$, then $H$ is bounded on $E^{(p)}$
for all $p>1$ and $\| H \|_{E^{(p)} \rightarrow E^{(p)}} \leq \| H \|_{E
\rightarrow E}^{1/p}$.
\end{itemize}

\noindent In the case when $E$ is a Banach ideal space, then we also have

\begin{itemize}
\item[$(iv)$] $H$ is bounded on $E$ if and only if $H^{*}$ is bounded on $%
E^{\prime }$ and $\| H \|_{E \rightarrow E} = \| H^{*} \|_{E^{\prime}
\rightarrow E^{\prime}}$.

\item[$(v)$] $D_{p}$ is bounded on $E$ if and only if $D_{1/p}$ is bounded
on $E^{\prime }$ for each $p>0$ and $\Vert D_{p}\Vert _{E\rightarrow
E}=\Vert D_{1/p}\Vert _{E^{\prime }\rightarrow E^{\prime }}$.
\end{itemize}
\end{remark}

\begin{proof} (i) It follows by the definition. (ii) By the assumption, 
$D_{2}$ is bounded on $E^{1/2}F^{1/2}$ (see \cite[Theorem 15.13, p.190]{Ma89}). 
Since $E\odot F=(E^{1/2}F^{1/2})^{(1/2)}$ (see \cite[Theorem 1]{KLM14}),
the conclusion follows by (i). (iii) Since $H$ is bounded on $L^{\infty }$,
it is enough to apply the equality $E^{(p)}=E^{1/p}(L^{\infty })^{1-1/p}$.
(iv) The necessity follows from the equality $\int (H^{\ast }x)y = \int xHy$
for each $x\in E^{\prime}$ and $y \in E$. (v) The proof comes from the
definition of the K\"{o}the dual.
\end{proof}

The paper is organized as follows: In Section 2 we define symmetrization $%
E^{(\ast )}$ of a quasi-normed ideal space $E$ on $I=(0,1)$ or $I=(0,\infty
) $ and collect some preliminary properties.

In Section 3 we investigate a commutativity property of symmetrization
operation $E\mapsto E^{\left( \ast \right) }$ with some known constructions,
like the sum of the spaces $E+F$, the Calder\'{o}n-Lozanovski\u{\i}
construction $\rho (E,F)$, the pointwise product $E\odot F$ and the K\"{o}the
duality $E^{\prime }$. In Theorem \ref{CLcommutstar}, we found conditions
under which $\rho (E,F)^{(\ast )}=\rho (E^{(\ast )},F^{(\ast )})$, in
particular, when $(E+F)^{(\ast )}=E^{(\ast )}+F^{(\ast )}$. Then, in Theorem %
\ref{dual-comut}, we prove that $(E^{\prime })^{(\ast )}=(E^{(\ast
)})^{\prime }$ under additional assumption on $E,$ which is essential (see
Example \ref{Ex4}), that is, the K\"{o}the duality does not commute with
symmetrization, in general for Banach ideal spaces.

In Section 4 we give sufficient conditions under which the space of
pointwise multipliers $M(E,F)$ commutes with the symmetrization operation $%
E\mapsto E^{\left( \ast \right) },$ that is $M(E,F)^{\left( \ast \right)
}=M(E^{\left( \ast \right) },F^{\left( \ast \right) })$ (Theorem \ref%
{multip-komut}). We also fully identify the space of pointwise multipliers
for classical Lorentz spaces $M(L^{p_{1},q_{1}},L^{p_{2},q_{2}})$ (Theorem \ref{Lorentz-multi}).

In Section 5 the notion of the explicit factorization for product space $E
\odot F$ is introduced. In Theorem 5 we proved that under some assumptions
on quasi-Banach ideal spaces $E, F$ from the explicit factorization for $G =
E \odot F$ it follows equality for symmetrizations $G^{(*)} = E^{(*)} \odot
F^{(*)}$ and the explicit factorization holds.

In section 6 we prove that, under some assumptions, from the factorization $%
F=E\odot M\left( E,F\right) $ we can conclude the factorization of
respective symmetrizations $F^{(\ast )}=E^{(\ast )}\odot M\left( E^{(\ast
)},F^{(\ast )}\right)$.

Finally, in Section 7, the space of multipliers and factorization of Ces\`{a}%
ro spaces is presented in Theorem \ref{mult-cesaro} and Corollary \ref%
{faktor-Cesaro}.

%%%%%%%%%%%%%%%%%%%%%%%%%%%%%%%%%%%% Section 2

\section{Symmetrization of quasi-normed ideal spaces}

%%%%%%%%%%%%%%%%%%%%%%%%%%%%%%%%%

Let $E = \left( E,\| \cdot \|_{E}\right) $ be a \textit{quasi-normed ideal
space} on $I$. The \textit{symmetrization} $E^{(\ast )}$ of $E$ is defined
as 
\begin{equation}  \label{sym}
E^{(\ast )} = \{x \in L^0(I) \colon x^* \in E \}
\end{equation}
with the functional $\| x\|_{E^{(\ast )}} = \| x^\ast \|_E$. For two
quasi-normed ideal spaces $E, F$ on $I$ it follows directly from the
definition that: 
\begin{equation}  \label{equality4}
E\overset{C}{\hookrightarrow }F ~ \mathrm{implies} ~~ E^{(*)} \overset{C}{%
\hookrightarrow }F^{(*)} ~~\mathrm{and} ~~ (E \cap F)^{(*)} = E^{(*)} \cap
F^{(*)}.
\end{equation}

It may happen that symmetrization is trivial, that is, $E^{(\ast )} = \{0\}$%
, as it is for example if $E = L^1(1/t)$ or $E= L^{\infty}(1/t)$. It is easy
to see that $E^{(\ast )} \neq \{0\}$ if and only if $\chi_{(0, a)} \in E$
for some $a > 0$.

A. Kami\'{n}ska and Y. Raynaud proved in \cite[Lemma 1.4]{KR09} that the
functional $\Vert \cdot \Vert _{E^{(\ast )}}$ is a quasi-norm if and only if
there is a constant $1\leq A<\infty $ such that 
\begin{equation}
\Vert D_{2}x^{\ast }\Vert _{E}\leq A\,\Vert x^{\ast }\Vert _{E}~~\mathrm{%
for~all}~~x^{\ast }\in E,  \label{dil}
\end{equation}%
and then $(E^{(\ast )},\Vert \cdot \Vert _{E^{(\ast )}})$ is a quasi-normed
symmetric space. The smallest possible constant $A$ in (\ref{dil}) we denote
by $A_{E}$. Furthermore, if $\left( E,\Vert \cdot \Vert _{E}\right) $ is a
quasi-Banach space, then $(E^{(\ast )},\Vert \cdot \Vert _{E^{(\ast )}})$ is
quasi-Banach too (see \cite[Lemma 1.4]{KR09}).

Condition (\ref{dil}) means that the dilation operator $D_{2}$ is bounded on
the cone of nonnegative nonincreasing elements $x=x^{\ast }\in E$. It
implies in particular that $E^{(\ast )}$ is a linear space which follows
immediately from the inequality $\left( x+y\right) ^{\ast }\left( t\right)
\leq x^{\ast }\left( t/2\right) +y^{\ast }\left( t/2\right) .$ However, it
is not obvious whether the linearity of $E^{(\ast )}$ gives the condition (%
\ref{dil}). We show that these conditions are in fact equivalent. Denote by%
\begin{equation*}
E^{\downarrow }\text{ -- the cone of nonnegative and nonincreasing elements }%
x=x^{\ast }\in E.\text{ }
\end{equation*}

\begin{lemma}
Let $E$ be a \textit{quasi-normed ideal space} on $I$. Then $E^{(\ast )}$ is
a linear space if and only if $D_{2}x\in E^{\downarrow }$ for each $x\in
E^{\downarrow }.$
\end{lemma}

\proof

The sufficiency follows from the inequality $\left( x+y\right) ^{\ast
}\left( t\right) \leq x^{\ast }\left( t/2\right) +y^{\ast }\left( t/2\right)
.$ We prove the necessity. Suppose there is an element $x\in E^{\downarrow }$
such that $D_{2}x\notin E^{\downarrow }.$ Thus, $x=x^{\ast }\in E$ and $%
\left( D_{2}x\right) ^{\ast }=D_{2}x,$ whence $D_{2}x\notin E.$ We consider
two cases.

$\left( 1\right) $ Let $I=\left( 0,\infty \right) .$ Set $%
A=\cup_{k=0}^{\infty }(2k,2k+1]$ and $A^{\prime }=\left( 0,\infty \right)
\backslash A.$ Let $\sigma _{1}:A\rightarrow \left( 0,\infty \right) $ and $%
\sigma _{2}:A^{\prime }\rightarrow \left( 0,\infty \right) $ be measure
preserving transformations. Define 
\begin{equation*}
x_{1}=x\circ \sigma _{1}\text{ and }x_{2}=x\circ \sigma _{2}.
\end{equation*}%
Then $x_{1},x_{2}\in E^{(\ast )},$ because $x_{1}^{\ast }=x_{2}^{\ast
}=x^{\ast }=x\in E.$ Moreover, since $x_{1}\bot x_{2},$ so $%
d_{x_{1}+x_{2}}=2d_{x}$ and $\left( x_{1}+x_{2}\right) ^{\ast }\left(
t\right) =x^{\ast }\left( t/2\right) =D_{2}x\notin E.$ It means that $%
E^{(\ast )}$ is not a linear space.

$\left( 2\right) $ Assume that $I=\left( 0,1\right) .$ Let $%
x_{1}=D_{1/2}D_{2}x.$ Then $\mathrm{{supp} \hspace{1mm} x_{1}\subset \left(
0,1/2\right) }$ and $x_{1}=x_{1}^{\ast }\leq x\in E.$ Moreover, $%
D_{2}x_{1}=D_{2}D_{1/2}D_{2}x=D_{2}x\notin E.$ Take%
\begin{equation*}
x_{2}\left( t\right) =\left\{ 
\begin{array}{ccc}
x_{1}\left( t-1/2\right) & \text{for} & t\in (1/2,1), \\ 
0 & \text{if} & t\in \left( 0,1/2\right) .%
\end{array}%
\right.
\end{equation*}%
Thus $x_{2}^{\ast }=x_{1}^{\ast }\in E$ and consequently $x_{1},x_{2}\in
E^{(\ast )}.$ On the other hand, $\left( x_{1}+x_{2}\right) ^{\ast }\left(
t\right) =x^{\ast }\left( t/2\right) =D_{2}x\notin E.$ Thus again $E^{(\ast
)}$ is not a linear space.\endproof

\begin{lemma}
Assume that $E$ is a \textit{quasi-normed ideal space} on $I.$ If $D_{2}x\in
E^{\downarrow }$ for each $x\in E^{\downarrow },$ then there is a constant $%
1\leq A<\infty $ such that $\Vert D_{2}x\Vert _{E}\leq A\,\Vert x\Vert _{E}~~%
\mathrm{for~all}~~x\in E^{\downarrow }$.
\end{lemma}

\proof Suppose the condition is not satisfied, that is, we find a sequence $%
\left( x_{n}\right) $ in $E^{\downarrow }$ such that $\Vert x_{n}\Vert
_{E}=1 $ and $\Vert D_{2}x_{n}\Vert _{E}\geq n\left( 2C\right) ^{n},$ where
the constant $C$ is from the quasi-triangle inequality of $E$. Let $%
y=\sum_{n=1}^{\infty }\left( 2C\right) ^{-n}x_{n}.$ Since $%
\sum_{n=1}^{\infty }C^{n}\left\Vert \left( 2C\right) ^{-n}x_{n}\right\Vert
<\infty ,$ by Theorem 1.1 from \cite{Ma04}, we conclude that $y\in E.$
Obviously, $y\in E^{\downarrow }.$ Furthermore, 
\begin{equation*}
D_{2}y=D_{2}\left( \sum_{n=1}^{\infty }\left( 2C\right) ^{-n}x_{n}\right)
\geq D_{2}\left( \left( 2C\right) ^{-n}x_{n}\right) =\left( 2C\right)
^{-n}D_{2}x_{n}\text{ for each }n.
\end{equation*}%
Thus $\Vert D_{2}y\Vert _{E}\geq \left( 2C\right) ^{-n}\Vert D_{2}x_{n}\Vert
_{E}\geq \left( 2C\right) ^{-n}n\left( 2C\right) ^{n}=n$ for each $n.$ Since 
$\left( D_{2}y\right) ^{\ast }=D_{2}y,$ it follows that $D_{2}y\notin
E^{\downarrow }.$\endproof

Now we are ready to conclude a stronger and more complete characterization
than it has been presented in \cite[Lemma 1.4]{KR09}.

\begin{corollary}
\label{quasi-liniowosc}Let $E$ be a \textit{quasi-normed ideal space} on $I$%
. The following statements are equivalent:\newline
$\left( i\right) $ $E^{(\ast )}$ is a linear space.\newline
$\left( ii\right) $ For each $x\in E^{\downarrow }$ we have $D_{2}x\in
E^{\downarrow }.$\newline
$\left( iii\right) $ There is a constant $1\leq A<\infty $ such that $\Vert
D_{2}x\Vert _{E}\leq A\,\Vert x\Vert _{E}~~\mathrm{for~all}~~x\in
E^{\downarrow }$.\newline
$\left( iv\right) $ $(E^{(\ast )},\Vert \cdot \Vert _{E^{(\ast )}})$ is a
quasi-normed space.
\end{corollary}

\proof The equivalences $\left( i\right) \Leftrightarrow \left( ii\right)
\Leftrightarrow \left( iii\right) $ come from the above two lemmas. The last
equivalence $\left( iii\right) \Leftrightarrow \left( iv\right) $ has been
proved in \cite[Lemma 1.4]{KR09}.\endproof

It is worth to mention that the equivalence $\left( i\right) \Leftrightarrow
\left( iv\right) $ for the Lorentz spaces $\Lambda _{p,w^{p}}$ is already
known. Namely, each of conditions $\left( i\right) $, $\left( iv\right) $
is equivalent to $W\in \Delta _{2}$ which has been proved in \cite{CKMP} and 
\cite{KM04}, respectively (see the disscusion following Problem \ref{P1} in
Section 4 for the respective definitions).

Of course, if $E^{\left( \ast \right) }=\{0\},$ then condition (\ref{dil})
is satisfied trivially. We present some examples when $E^{\left( \ast
\right) }\neq \{0\}$ and condition (\ref{dil}) does not hold (equivalently
none of conditions from\ Corollary \ref{quasi-liniowosc} is satisfied).

%%%%%%%%%%%%%%%% Example 1

\begin{example}
\label{Ex1} \rm{(a) Consider $E=L^{\infty }\left( \frac{1}{1-t}\right) $
on $I=(0,1)$. Then, taking $x=\chi _{\left( 0,1/2\right) }$ we have $%
x=x^{\ast }\in E,$ whence $E^{\left( \ast \right) }\neq \{0\}$. Moreover, $%
D_{2}x=\chi _{\left( 0,1\right) }\notin E$. }

\rm{(b) For $a>0$ let 
\begin{equation*}
w_{a}(t)=\chi _{(0,a)}(t)+\frac{1}{t-a}\chi _{(a,\infty )}(t),~t>0.
\end{equation*}%
Consider the weighted Banach ideal space $E=L^{1}(w_{a})$ on $I=(0,\infty )$
and its symmetrization $E^{(\ast )}=[L^{1}(w_{a})]^{(\ast )}=\Lambda
_{1,w_{a}}$. Of course, $E^{(\ast )}\neq \{0\}$ since $\chi _{(0,a)}\in
E^{(\ast )}$ with $\Vert \chi _{(0,a)}\Vert _{E^{(\ast )}}=a$. On the other
hand, 
\begin{equation*}
\Vert D_{2}\chi _{(0,a)}\Vert _{E}=a+\int_{a}^{2a}\frac{1}{t-a}\,dt=\infty ,
\end{equation*}%
which means that $\Vert \cdot \Vert _{E^{(\ast )}}$ is not a quasi-norm (see 
\cite[Lemma 1.4]{KR09}), equivalenty, none of conditions from\ Corollary \ref%
{quasi-liniowosc} is satisfied. For example, $W_{a}(t)=%
\int_{0}^{t}w_{a}(s)ds=t$ for $0<t<a$ and $W_{a}(t)=\infty $ for $t>a$. Thus 
$W_{a}\notin \Delta _{2}$ and by Remark 1.3 in \cite{CKMP} we get that $%
E^{(\ast )}=\Lambda _{1,w_{a}}$ is not a linear space. In fact, $x=\chi _{(0,%
\frac{a}{2})}\in \Lambda _{1,w_{a}},y=\chi _{(\frac{a}{2},\frac{5a}{4})}\in
\Lambda _{1,w_{a}}$, but $x+y=x=\chi _{(0,\frac{5a}{4})}\notin \Lambda
_{1,w_{a}}$. }

\rm{(c) Let $E=L^{1}(w)$ with $w(t)=e^{t}$ on $I=(0,\infty )$. Then $%
W(t)=\int_{0}^{t}w(s)ds=e^{t}-1<\infty $ for any $t\in I$. Clearly, $%
E^{(\ast )}\neq \{0\}$ but 
\begin{equation*}
\frac{\Vert D_{2}\chi _{(0,n)}\Vert _{E}}{\Vert \chi _{(0,n)}\Vert _{E}}=%
\frac{e^{2n}-1}{e^{n}-1}\rightarrow \infty ~~as~~n\rightarrow \infty .
\end{equation*}
}
\end{example}

%%%%%%%%%%%%%

%%%%%%%%%%%%%%%% Lemma 1
The following lemma completes the above disscusion of the case $E^{(\ast
)}\neq \{0\}$.

\begin{lemma}
\label{nonzero} Let $E$ be a quasi-Banach ideal space $E$ on $I$ such that
the dilation operator $D_{2}$ is bounded on $E^{\downarrow }$.

\begin{itemize}
\item[$(i)$] If $E^{( \ast)}\neq \{0\}$, then $\chi_{(0, a)} \in E$ for each 
$a > 0$ and $E^{( \ast)}$ has a weak unit.

\item[$(ii)$] Let $I = (0, 1)$. Then $E^{( \ast)}\neq \{0\}$ if and only if $%
L^{\infty }\hookrightarrow E$.

\item[$(iii)$] Let $I = (0,\infty )$. Then $E^{( \ast) }\neq \{0\}$ if and
only if $L_{b}^{\infty }\hookrightarrow E$, where $L_{b}^{\infty }$ is the
linear subspace of $L^{\infty }$ consisting of all essentially bounded
functions with bounded support.
\end{itemize}
\end{lemma}

%%%%%%%%%%%%%
\begin{proof}
$(i)$ The weak unit in $E^{(*)}$ can be given by $x_0 = \sum_{n=1}^{\infty} x_n$ with $x_n = \dfrac{\chi_{(n-1, n)}}{b_n \| \chi_{(n-1, n)}\|_E}$, where $b_n$'s 
are chosen so that the sequence $\{b_n \| \chi_{(n-1, n)}\|_E \}$ is increasing and $\sum_{n=1}^{\infty} 1/b_n < \infty$ 
(cf. \cite{Ko15}).

In both cases $(ii)$ and $(iii)$ only the necessity need to be proved.
\newline
$(ii)$ By the assumption, there is $a > 0$ with $x=\chi _{(0, a]}\in E$. 
Consequently, we find $k\in \mathbb{N}$ %\mathbb%
such that $\chi _{(0, 1)} =  D_{2}^{k} x  \in E$, which proves the conclusion.
\newline
$(iii) $ Similarly as above we conclude that $\chi_{(0,b)}\in E$ for each $b > 0$ and we are done.
\end{proof}

The symmetrization $E^{(\ast )}$ of a Banach ideal space $E$ has been
intensively studied recently (cf. \cite{Fo11}, \cite{Fo12}, \cite{KM07}, 
\cite{KR09}, \cite{Ko15}, \cite{Ko16}). The Lorentz, Marcinkiewicz and
Orlicz--Lorentz spaces are particular cases of this construction. In fact,
the symmetrization $[L^{p}(w)]^{(\ast )}$ of weighted Lebesgue spaces $%
L^{p}(w)$ even for $0<p<\infty $ is the Lorentz space $\Lambda _{p,w^{p}}$,
which structure was investigated in \cite{CKMP}, \cite{KM04}, the
symmetrization $[L^{\infty }(w)]^{(\ast )}$ of the weighted space $L^{\infty
}(w)$ is the Marcinkiewicz space $M_{w}$, and the symmetrization $(L^{\Phi
})^{(\ast )}$ of the Musielak--Orlicz spaces $L^{\Phi }$ with $\Phi
(t,u)=\varphi (u)w(t)$ is the Orlicz--Lorentz space $\Lambda _{\varphi ,w}$,
which structure was and still is investigated in many papers (cf. \cite{KR09}
and literature therein).

Note that if $E$ has the Fatou property, then so it has its symmetrization $%
E^{(\ast )}$ since $0\leq x_{n}\uparrow x$ implies $x_{n}^{\ast }\uparrow
x^{\ast }$. Also if $E$ is order continuous and does not contain the
function $\chi _{(0,\infty )}$, then $E^{(\ast )}$ is order continuous (cf. 
\cite[p. 279]{KR09}). The precise characterization for $x\in E^{(\ast )}$ to
be a point of order continuity of $E^{(\ast )}$ has been given in \cite[%
Theorem 3.9]{Ko16}. The local approach to monotonicity properties of $%
E^{(\ast )}$ has been presented in \cite[Theorem 3.6 and 3.8]{Ko16}.

%%%%%%%%%%%%%%%%%%%% Remark 1

\begin{remark}
\label{Hr} Let $E$ be a quasi-Banach ideal space on $I$ and $E^{\left( \ast
\right) }\neq \left\{ 0\right\} .$ If $\Vert H_{r}x^{\ast }\Vert _{E}\leq
C\,\Vert x^{\ast }\Vert _{E}$ for all $x^{\ast }\in E$, then $\Vert
D_{2}x^{\ast }\Vert _{E}\leq 2^{1/r}C\,\Vert x^{\ast }\Vert _{E}$ for all $%
x^{\ast }\in E$. In particular, if the operator $H_{r}$ is bounded on $E$,
then (\ref{dil}) holds with $A_{E}\leq 2^{1/r}\Vert H_{r}\Vert
_{E\rightarrow E}$.
\end{remark}

%%%%%%%%%%%%%%%%
\proof Indeed, for $t\in I$, we have 
\begin{equation*}
H_{r}x^{\ast }(t)=(\int_{0}^{1}x^{\ast }(st)^{r}ds)^{1/r}\geq
(\int_{0}^{1/2}x^{\ast }(st)^{r}ds)^{1/r}\geq x^{\ast }(t/2)\,2^{-1/r},
\end{equation*}%
and so $\Vert H_{r}x^{\ast }\Vert _{E}\geq 2^{-1/r}\Vert D_{2}x^{\ast }\Vert
_{E}$.\endproof

Below we can see that a similar result with the operator $H_{r}^{\ast }$ is
not true.

%%%%%%%%%%%%%%%% Example 2

\begin{example}
\label{Ex2} \rm{\ We show that $H_{r}^{\ast }$ is bounded on $E=L^{r}(w)
$ with $0<r<\infty ,w(t)=e^{t}$ on $I=[0,\infty )$, but estimate (\ref{dil})
does not hold. Namely, we have 
\begin{eqnarray*}
\Vert H_{r}^{\ast }x\Vert _{L^{r}(w)}^{r} &=&\int_{0}^{\infty }\left(
\int_{t}^{\infty }\frac{|x(s)|^{r}}{s}ds\right) e^{tr}dt=\int_{0}^{\infty
}\left( \int_{0}^{s}e^{tr}dt\right) \frac{|x(s)|^{r}}{s}ds \\
&=&\int_{0}^{\infty }\frac{e^{sr}-1}{sr}|x(s)|^{r}ds\leq \int_{0}^{\infty
}e^{sr}|x(s)|^{r}ds=\Vert x\Vert _{L^{r}(w)}^{r}
\end{eqnarray*}%
and 
\begin{equation*}
\frac{\Vert D_{2}\chi _{(0,a)}\Vert _{L^{r}(w)}^{r}}{\Vert \chi
_{(0,a)}\Vert _{L^{r}(w)}^{r}}=\frac{e^{2ar}-1}{e^{ar}-1}\rightarrow \infty
~~as~~a\rightarrow \infty .
\end{equation*}%
}
\end{example}

%%%%%%%%%%%%%%%%%%%% Remark 2

\begin{remark}
\label{bounded} The operator $H_r$ (or $H_r^{\ast}$) is bounded on a
quasi-Banach ideal space $E$ if and only if the operator $H$ (or $H^{\ast}$)
is bounded on $E^{(1/r)}$. Moreover, 
\begin{equation*}
\| H_r\|_{E \rightarrow E} = \| H\|_{E^{(1/r)} \rightarrow E^{(1/r)}}^{1/r}
~~ \mathrm{and} ~~\| H_r^{\ast} \|_{E \rightarrow E} = \|
H^{\ast}\|_{E^{(1/r)} \rightarrow E^{(1/r)}}^{1/r}.
\end{equation*}
\end{remark}

\begin{remark}
Let $E$ be a quasi-Banach ideal space. If $H$ is bounded on $E$, then $H$ is
bounded on $E^{(\ast )}$ and $E^{(\ast )}$ has the majorant property.
\end{remark}

\proof Since $H$ is bounded on $E,$ by Remark \ref{Hr} and Corollary \ref%
{quasi-liniowosc}, we conclude that $E^{(\ast )}$ is a quasi-normed space.
Let $x\in E^{(\ast )}$. By the assumption we obtain 
\begin{equation*}
\left\Vert Hx\right\Vert _{E^{(\ast )}}=\left\Vert (Hx)^{\ast }\right\Vert
_{E}\leq \left\Vert (Hx^{\ast })^{\ast }\right\Vert _{E}=\left\Vert Hx^{\ast
}\right\Vert _{E}\leq C\left\Vert x^{\ast }\right\Vert _{E}=C\left\Vert
x\right\Vert _{E^{(\ast )}}.
\end{equation*}%
We prove the majorant property. Let $x\in L^{0},\ y\in E^{(\ast )}$ and $%
Hx^{\ast }\leq Hy^{\ast }$. We need to prove that $x\in E^{(\ast )}$ or
equivalently $x^{\ast }\in E$. But this follows from $x^{\ast }\leq Hx^{\ast
}\leq Hy^{\ast }\in E.$\endproof

%%%%%%%%%%%%%%%%

%%%%%%%%%%%%%%%% Section 3   %%%%%%%%%%%%%%%%%%%%%%%%%%

\section{Symmetrization of some known constructions}

%%%%%%%%%%%%%%%%%%%%%%%%%%%

Consider two quasi-Banach ideal spaces $E,F$ on $I$. The symmetrization
commutes with the intersection $(E\cap F)^{(\ast )}=E^{(\ast )}\cap F^{(\ast
)}$ and from (\ref{equality4}) we conclude that 
\begin{equation*}
E^{(\ast )}+F^{(\ast )}\hookrightarrow (E+F)^{(\ast )}.
\end{equation*}%
We can then ask what about commutativity of the symmetrization operation $%
E\mapsto E^{\left( \ast \right) }$ with the sum or more general with the
Calder\'{o}n--Lozanovski{\u{\i}} construction?

Let us therefore investigate the symmetrization of the Calder\'{o}%
n--Lozanovski{\u{\i}} construction.

First, note that for two symmetric spaces $E$ and $F$ the %\mathcal{P}%
Calder\'{o}n--Lozanovski{\u{\i}} space $\rho (E,F)$ for any $\rho \in 
\mathcal{P}$ is also a symmetric space up to an equivalence of quasi-norms
(using Lemma 4.3 from \cite[p. 93]{KPS82} and $\rho (E,F)$ is an
interpolation space between $E$ and $F$ for positive operators -- cf. \cite[%
Theorem 15.13 on p. 190]{Ma89}). Consequently, 
\begin{equation*}
\rho (E,F)^{(\ast )}=\rho (E,F)=\rho (E^{(\ast )},F^{(\ast )})
\end{equation*}%
In \cite{KLM14} it has been proved that the Calder\'{o}n construction $%
E^{\theta }F^{1-\theta }$ commutes with the symmetrization operation $%
E\mapsto E^{\left( \ast \right) }$ for Banach ideal spaces $E,F$ (see \cite[%
Lemma 4]{KLM14}).

Now, we generalize this result to the case of Calder\'{o}n--Lozanovski{\u{\i}%
} construction $\rho (E,F)$ and for quasi-Banach ideal spaces $E,F$. This
also shows that for the Calder\'{o}n construction $E^{\theta }F^{1-\theta }$
we may use much weaker assumptions about the spaces than in \cite{KLM14}.

%%%%%%%%%%%%%%%%%%%%%%%%%%%%%% Thm 1
In the following theorem we consider two constructions. Below the
assumptions on $D_{2}$ operator imply that space $\rho (E^{(\ast )},F^{(\ast
)})$ is a linear, quasi-normed space. By Corollary \ref{quasi-liniowosc},
the space $\rho (E,F)^{(\ast )}$ is linear if and only if the operator $%
D_{2} $ is bounded on $\rho (E,F)^{\downarrow },$ which is not a
\textquotedblleft nice\textquotedblright\ assumption. This is a motivation
and a reason for the formulation below. First we prove only inclusions
between sets $\rho (E^{(\ast )},F^{(\ast )})$ and $\rho (E,F)^{(\ast )}$ and
the respective inequalities (\ref{8}) and (\ref{9}). Next, we conclude the
equality between sets $\rho (E^{(\ast )},F^{(\ast )})=\rho (E,F)^{(\ast )},$
which implies that the space $\rho (E,F)^{(\ast )}$ is in fact a linear,
quasi-normed space. For the same reasons we formulate in a special way
Corollary \ref{ilocz-gw-komut} and Theorem \ref{dual-comut}.

\begin{theorem}
\label{CLcommutstar} Let $E$ and $F$ be quasi-Banach ideal spaces such that $%
E^{(\ast )}\neq \{0\},$ $F^{(\ast )}\neq \{0\}$ and the operator $D_{2}$ is
bounded both on $E^{\downarrow }$ and $F^{\downarrow }$ -- see Corollary \ref%
{quasi-liniowosc} for equivalent conditions. Then:

\begin{itemize}
\item[$(i)$] The Calder\'{o}n--Lozanovski{\u{\i}} construction $\rho
(E^{(\ast )},F^{(\ast )})\neq \{0\},$ $\rho (E^{(\ast )},F^{(\ast )})\subset
\rho (E,F)^{(\ast )}$ and 
\begin{equation}
\left\Vert x\right\Vert _{\rho (E,F)^{(\ast )}}\leq C_{1}\left\Vert
x\right\Vert _{\rho (E^{(\ast )},F^{(\ast )})}~~\text{for all }x\in \rho
(E^{(\ast )},F^{(\ast )})\text{ }  \label{8}
\end{equation}%
with$~~C_{1}\leq \max (A_{E},A_{F}),$ where $A_{E},A_{F}$ are the best
constants in (\ref{dil}).

\item[$(ii)$] If, additionally, the operator $H_{r}^{\ast }$ is bounded on
the spaces $E,F$ for some $r>0$, then $\rho (E,F)^{(\ast )}\subset \rho
(E^{(\ast )},F^{(\ast )})$ and%
\begin{equation}
\left\Vert x\right\Vert _{\rho (E^{(\ast )},F^{(\ast )})}\leq
C_{2}\left\Vert x\right\Vert _{\rho (E,F)^{(\ast )}}~~\text{for all }x\in
\rho (E,F)^{(\ast )}  \label{9}
\end{equation}%
with $C_{2}\leq 2^{1/r}\max (1,2^{1/r-1})\cdot \max \left( A_{E}\Vert
H_{r}^{\ast }\Vert _{E\rightarrow E},A_{F}\Vert H_{r}^{\ast }\Vert
_{F\rightarrow F}\right) .$\newline
In particular, the inequalities (\ref{8}) and (\ref{9}) imply that the
functional $\left\Vert \cdot \right\Vert _{\rho (E,F)^{(\ast )}}$ is a
quasi-norm on the space $\rho (E,F)^{(\ast )}$ and%
\begin{equation}
\rho (E,F)^{(\ast )}=\rho (E^{(\ast )},F^{(\ast )}).  \label{9.1}
\end{equation}
\end{itemize}
\end{theorem}

%%%%%%%%%%%%%%%%%%%%%%%%

%TCIMACRO{\TeXButton{proof}{\proof}}%
%BeginExpansion
\proof%
%EndExpansion
$(i)$ Since $E^{(\ast )}\neq \{0\}$ and $F^{(\ast )}\neq \{0\}$ are
quasi-normed spaces it follows, by the Aoki--Rolewicz theorem, that there
are $p_{0}$-norm $\Vert \cdot \Vert _{1}$ on $E^{(\ast )}$ and $p_{1}$-norm
on $F^{(\ast )}$. But $E^{(\ast )},F^{(\ast )}$ are symmetric $p_{i}$-normed
spaces, which gives inclusions $L^{p_{0}}\cap L^{\infty }\overset{C_{3}}{%
\hookrightarrow }E^{(\ast )},L^{p_{1}}\cap L^{\infty }\overset{C_{4}}{%
\hookrightarrow }F^{(\ast )}$, with $C_{3}\leq 2^{1/p_{0}}\Vert \chi
_{\lbrack 0,1]}\Vert _{E}$ and $C_{4}\leq 2^{1/p_{1}}\Vert \chi _{\lbrack
0,1]}\Vert _{F}$ (see \cite[Theorem 1]{BHS91}).

Clearly, if $E_1 \overset{B}{\hookrightarrow } E_2$ with $B \geq 1$, then $%
\rho (E_1, F) \overset{B}{\hookrightarrow } \rho (E_2, F)$ and similarly for
the inclusion with respect to the second variable. Since for $p = \min (p_0,
p_1)$ we have $L^p \cap L^{\infty} \overset{1}{\hookrightarrow } L^{p_i}
\cap L^{\infty}$, \, ($i = 0, 1$), thus putting all above information
together we obtain 
\begin{equation*}
L^p \cap L^{\infty} = \rho (L^p \cap L^{\infty}, L^p \cap L^{\infty}) {%
\hookrightarrow } \rho (L^{p_0} \cap L^{\infty}, L^{p_1} \cap L^{\infty}) {%
\hookrightarrow } \rho(E^{(\ast)}, F^{(\ast)}),
\end{equation*}
which means that the last space is nontrivial.

We will prove the inclusion $\rho (E^{(\ast )},F^{(\ast )})\subset \rho
(E,F)^{(\ast )}$. Let $x\in \rho (E^{(\ast )},F^{(\ast )})$. Then $|x|\leq
\lambda \rho \left( |x_{0}|,|x_{1}|\right) $ for some $\lambda >0$ and $%
\Vert x_{0}\Vert _{E^{(\ast )}}\leq 1,\Vert x_{1}\Vert _{F^{(\ast )}}\leq 1$%
. Recall that for a given function $\rho \in \mathcal{P}$ the function $\widehat{\rho }
$ defined by 
\begin{equation*}
\widehat{\rho }(a,b)=\inf_{u,v>0}\frac{au+bv}{\rho (u,v)}\,\mathrm{\ }\text{%
for\ all}\ a,b\geq 0
\end{equation*}%
belongs to $\mathcal{P}$ and this operation is an involution on $\mathcal{P}$, that is, $%
\widehat{\widehat{\rho }}=\rho $ (see \cite[Lemma 2]{Lo78} and also \cite[%
Lemma 15.8]{Ma89}). Since 
\begin{equation*}
\rho \left( |x_{0}(t)|,|x_{1}(t)|\right) \leq \frac{|x_{0}(t)|\,u+|x_{1}(t)|%
\,v}{\widehat{\rho }(u,v)}
\end{equation*}%
for all $u,v>0$, it follows that 
\begin{equation*}
\rho \left( |x_{0}|,|x_{1}|\right) ^{\ast }(t)\leq \frac{x_{0}^{\ast
}(t/2)\,u+x_{1}^{\ast }(t/2)\,v}{\widehat{\rho }(u,v)}.
\end{equation*}%
Taking infimum over all $u,v>0$, we get 
\begin{equation*}
\rho \left( |x_{0}|,|x_{1}|\right) ^{\ast }(t)\leq \widehat{\widehat{\rho }}%
(x_{0}^{\ast }(t/2),x_{1}^{\ast }(t/2))=\rho (x_{0}^{\ast }(t/2),x_{1}^{\ast
}(t/2)).
\end{equation*}%
Consequently, 
\begin{equation*}
x^{\ast }(t)\leq \lambda \rho (x_{0}^{\ast }(t/2),x_{1}^{\ast }(t/2))\leq
\lambda \max (A_{E},A_{F})\,\rho (\frac{x_{0}^{\ast }(t/2)}{A_{E}},\frac{%
x_{1}^{\ast }(t/2)}{A_{F}}),
\end{equation*}%
where $A_{E},A_{F}$ are the smallest constants in (\ref{dil}) for $E,F$,
respectively. This means that $x\in \rho (E,F)^{(\ast )}$ with the norm $%
\leq \lambda \max (A_{E},A_{F})$. Thus, $\rho (E^{(\ast )},F^{(\ast
)})\subset \rho (E,F)^{(\ast )}$ and inequality (\ref{8}) is proved.

$(ii)$ Assume now that $H_r^{*}$ is bounded on the spaces $E, F$ for some $r
> 0$. Then it is bounded on $\rho(E, F)$ since this construction is an
interpolation space between $E$ and $F$ for positive operators (see \cite[%
Theorem 15.13 on p. 190]{Ma89}).

Now, we will prove the reverse inclusion $\rho (E,F)^{(\ast )}\subset \rho
(E^{(\ast )},F^{(\ast )})$. Let $x\in \rho (E,F)^{(\ast )}$. Then $x^{\ast
}\in \rho (E,F)$ and $x^{\ast }\leq \lambda \rho \left(
|x_{0}|,|x_{1}|\right) $ for some $\lambda >0$ and $\Vert x_{0}\Vert
_{E}\leq 1,\Vert x_{1}\Vert _{F}\leq 1$. Note that%
\begin{equation*}
x^{\ast }\left( t\right) \leq 2^{1/r}\left( \int_{t/2}^{t}x^{\ast }\left(
s\right) ^{r}\frac{ds}{s}\right) ^{1/r}\leq 2^{1/r}\left( \int_{t/2}^{\infty
}x^{\ast }\left( s\right) ^{r}\frac{ds}{s}\right) ^{1/r}
\end{equation*}%
\begin{equation*}
=2^{1/r}H_{r}^{\ast }\left( x^{\ast }\right) \left( t/2\right)
=2^{1/r}D_{2}H_{r}^{\ast }\left( x^{\ast }\right) \left( t\right) .
\end{equation*}%
Consequently, 
\begin{equation*}
x^{\ast }\leq 2^{1/r}\lambda D_{2}H_{r}^{\ast }\left( \rho \left(
|x_{0}|,|x_{1}|\right) \right) .
\end{equation*}%
On the other hand, applying the \ definition of $\widehat{\rho }$ and the
equality $\widehat{\widehat{\rho }}=\rho $ (see the proof of $\left(
i\right) $), we get%
\begin{equation*}
H_{r}^{\ast }\rho (|x_{0}|,|x_{1}|)\leq H_{r}^{\ast }\left( \frac{%
|x_{0}|\,u+|x_{1}|\,v}{\widehat{\rho }(u,v)}\right) =\,\max (1,2^{1/r-1})%
\frac{(H_{r}^{\ast }|x_{0}|)\,u+(H_{r}^{\ast }|x_{1}|)\,v}{\widehat{\rho }%
(u,v)}
\end{equation*}%
for every $u,v>0$. Taking infimum over all $u,v>0$ we obtain 
\begin{eqnarray*}
H_{r}^{\ast }\rho \left( |x_{0}|,|x_{1}|\right) &\leq &\max (1,2^{1/r-1})%
\widehat{\widehat{\rho }}\left( (H_{r}^{\ast }|x_{0}|),(H_{r}^{\ast
}|x_{1}|)\right) \\
&=&\max (1,2^{1/r-1})\rho \left( (H_{r}^{\ast }|x_{0}|),(H_{r}^{\ast
}|x_{1}|)\right) .
\end{eqnarray*}%
Thus, 
\begin{eqnarray*}
x^{\ast } &\leq &\lambda \,2^{1/r}\max (1,2^{1/r-1})D_{2}\,\rho \left(
H_{r}^{\ast }|x_{0}|,H_{r}^{\ast }|x_{1}|\right) \\
&=&\lambda \,2^{1/r}\max (1,2^{1/r-1})\rho \left( D_{2}H_{r}^{\ast
}|x_{0}|,D_{2}H_{r}^{\ast }|x_{1}|\right) .
\end{eqnarray*}%
By Ryff's theorem there exists a measure preserving transformation:\newline
$\left( I\right) $ $\omega :I\rightarrow I$ such that $x^{\ast }\circ \omega
=\left\vert x\right\vert $ a.e. when $m\left( \mathrm{supp} \hspace{1mm}
x\right) <\infty ,$\newline
$\left( II\right) $ $\omega :\mathrm{supp} \hspace{1mm} x\rightarrow
(0,\infty )$ such that $x^{\ast }\circ \omega =\left\vert x\right\vert $
a.e. on $\mathrm{supp} \hspace{1mm} x$ when $m\left( \mathrm{supp} \hspace{%
1mm} x\right) =\infty $ under the additional assumption that $x^{\ast
}(\infty )=0$ (see \cite{BS88}, Theorem 7.5 for $I=(0,1)$ or Corollary 7.6
for $I=(0,\infty )$). Note that our assumptions imply that $x^{\ast }(\infty
)=0$ in the same way as in the proof of Lemma 4 from \cite{KLM14}.
Therefore, 
\begin{eqnarray*}
|x| &=&x^{\ast }(\omega )\leq \lambda \,2^{1/r}\max (1,2^{1/r-1})\,\rho 
\left[ (D_{2}H_{r}^{\ast }|x_{0}|)(\omega ),(D_{2}H_{r}^{\ast
}|x_{1}|)(\omega )\right] \\
&=&\lambda \,2^{1/r}\max (1,2^{1/r-1})\,\rho (u_{0},u_{1})
\end{eqnarray*}%
with $u_{0}=(D_{2}H_{r}^{\ast }|x_{0}|)(\omega )$ and $u_{1}=(D_{2}H_{r}^{%
\ast }|x_{1}|)(\omega ).$ Let us see that $u_{0}\in E^{(\ast )}$ and $%
u_{1}\in F^{(\ast )}$ with $\Vert u_{0}\Vert _{E^{(\ast )}}\leq
A_{E}\left\Vert H_{r}^{\ast }\right\Vert _{E\rightarrow E}$ and $\Vert
u_{1}\Vert _{F^{(\ast )}}\leq A_{F}\left\Vert H_{r}^{\ast }\right\Vert
_{F\rightarrow F}$. In fact, similarly as in the proof of Lemma 4 in \cite%
{KLM14}, $H_{r}^{\ast }\left\vert x_{i}\right\vert $ is a nonincreasing
function and so is $D_{2}H_{r}^{\ast }\left\vert x_{i}\right\vert \,(i=0,1)$%
, which gives that 
\begin{equation*}
D_{2}H_{r}^{\ast }\left\vert x_{i}\right\vert =[D_{2}H_{r}^{\ast }\left\vert
x_{i}\right\vert ]^{\ast }=[(D_{2}H_{r}^{\ast }\left\vert x_{i}\right\vert
)(\omega )]^{\ast }~~\text{for}~~i=0,1.
\end{equation*}%
Hence, 
\begin{eqnarray*}
\Vert u_{0}\Vert _{E^{(\ast )}} &=&\Vert u_{0}^{\ast }\Vert _{E}=\Vert
\lbrack (D_{2}H_{r}^{\ast }\left\vert x_{0}\right\vert )(\omega )]^{\ast
}\Vert _{E}=\Vert D_{2}H_{r}^{\ast }\left\vert x_{0}\right\vert \Vert _{E} \\
&\leq &A_{E}\Vert H_{r}^{\ast }\left\vert x_{0}\right\vert \Vert _{E}\leq
A_{E}\Vert H_{r}^{\ast }\Vert _{E\rightarrow E}\Vert x_{0}\Vert _{E}\leq
A_{E}\Vert H_{r}^{\ast }\Vert _{E\rightarrow E}
\end{eqnarray*}%
and 
\begin{eqnarray*}
\Vert u_{1}\Vert _{F^{(\ast )}} &=&\Vert u_{1}^{\ast }\Vert _{F}=\Vert
\lbrack (D_{2}H_{r}^{\ast }\left\vert x_{1}\right\vert )(\omega )]^{\ast
}\Vert _{F}=\Vert D_{2}H_{r}^{\ast }\left\vert x_{1}\right\vert \Vert _{F} \\
&\leq &A_{F}\Vert H_{r}^{\ast }\left\vert x_{1}\right\vert \Vert _{F}\leq
A_{F}\Vert H_{r}^{\ast }\Vert _{F\rightarrow F}\Vert x_{1}\Vert _{F}\leq
A_{F}\Vert H_{r}^{\ast }\Vert _{F\rightarrow F},
\end{eqnarray*}%
which means that $x\in \rho (E^{(\ast )},F^{(\ast )})$ with the norm $\leq
\lambda C_{2}$, where 
\begin{equation*}
C_{2}\leq 2^{1/r}\max (1,2^{1/r-1})\cdot \max \left( A_{E}\Vert H_{r}^{\ast
}\Vert _{E\rightarrow E},A_{F}\Vert H_{r}^{\ast }\Vert _{F\rightarrow
F}\right) .
\end{equation*}%
Thus, $\rho (E,F)^{(\ast )}\subset \rho (E^{(\ast )},F^{(\ast )})$ and
inequality (\ref{9}) is proved. Summing up cases $\left( i\right) $ and $%
\left( ii\right) $ we conclude that the functional $\left\Vert \cdot
\right\Vert _{\rho (E,F)^{(\ast )}}$ is a quasi-norm on the space $\rho
(E,F)^{(\ast )}$ and equality (\ref{9.1}) holds. 
%TCIMACRO{\TeXButton{proof}{\endproof}}%
%BeginExpansion
\endproof%
%EndExpansion

%%%%%%%%%%%%%%%%%

Without additional assumptions, like those in Theorem \ref{CLcommutstar}$%
\left( ii\right) $, we can have that $E^{(\ast )}+F^{(\ast
)}\not\hookrightarrow (E+F)^{(\ast )}$ even for Banach ideal spaces $E,F$.
Examples below are inspired by \cite[Examples 1 and 3]{CM96a} and \cite[%
Example 3]{CM96b}.

%%%%%%%%%%%%%%%% Example 3

\begin{example}
\label{Ex3} \emph{{\ (a) For $0 < a < b$ let $E = L^1(w_a)$ and $F =
L^1(w_b) $, where weights $w_a, w_b$ are as in Example 1(b). Then $E + F =
L^1(\min(w_a, w_b)), (E + F)^{(\ast)} = \Lambda_{1, \min(w_a, w_b)}$ and $%
\chi_{(0, c)} \in (E + F)^{(\ast)}$ for any $c > 0$ since $\min(w_a, w_b)
\leq \max(1, \frac{1}{b-a})$. } }

\emph{{If $c < b$, then $\chi_{(0, c)} \in E^{(\ast)} + F^{(\ast)}$ since
for the decomposition $\chi_{(0, c)} = \chi_{(0, a)} + \chi_{(a, c)}$ we
have $\chi_{(0, a)} \in E^{(\ast)}$ and $\chi_{(a, c)} \in F^{(\ast)}$. The
last fact follows from the observation that $\chi_{(a, c)}^{\ast} =
\chi_{(0, c - a)} \leq \chi_{(0, b)} \in L^1(w_b)$. } }

\emph{{If $c>b$, then $\chi _{(0,c)}\notin E^{(\ast )}+F^{(\ast )}$ since
any decomposition of $\chi _{(0,c)}$ into decreasing functions has the form $%
a\chi _{(0,c)}+(1-a)\chi _{(0,c)}$ with $0\leq a\leq 1$ and $\chi
_{(0,c)}\notin E^{(\ast )},\chi _{(0,c)}\notin F^{(\ast )}$. } }

\emph{{Therefore, $E^{(\ast )}+F^{(\ast )}\not\hookrightarrow (E+F)^{(\ast
)} $, but unfortunately, by Remark 3.1 in \cite{CKMP}, the spaces $E^{(\ast
)}$ and $F^{(\ast )}$ are not linear since $W_{a}(t)=\infty $ for $t>a$ and $%
W_{b}(t)=\infty $ for $t>b.$ Moreover, $W(t)=\int_{0}^{t}\min
(w_{a}(s),w_{b}(s))ds<\infty $ for $t>0$. } }

\emph{{(b) We give now examples of linear spaces $E^{(\ast )},F^{(\ast
)},(E+F)^{(\ast )}$ for which we still have only proper inclusion $E^{(\ast
)}+F^{(\ast )}\not\hookrightarrow (E+F)^{(\ast )}$. For $0<w\in
L^{1}(0,\infty ),w$ decreasing, continuous and $w(t)\leq 1$ for all $t>0$,
let 
\begin{equation*}
w_{0}=\sum_{n=0}^{\infty }w\chi _{(2n,2n+1)}+\sum_{n=0}^{\infty }\chi
_{(2n+1,2n+2)}~~and~~w_{1}=\sum_{n=0}^{\infty }\chi
_{(2n,2n+1)}+\sum_{n=0}^{\infty }w\chi _{(2n+1,2n+2)}.
\end{equation*}%
Take $E=L^{1}(w_{0}),F=L^{1}(w_{1})$ and $x=\chi _{(0,\infty )}$. Then $x\in
(E+F)^{(\ast )}$, but $x\notin E^{(\ast )}+F^{(\ast )}$ because for any
decomposition $x=x_{0}+x_{1}$ we have $m(A_{i})=\infty $ for at least one $i$%
, where $A_{i}=\{t\in I\colon x_{i}(t)\geq 1/2\}$. Thus, $E^{(\ast
)}+F^{(\ast )}\not\hookrightarrow (E+F)^{(\ast )}$. } }

\emph{{On the other hand, the functions $W_{i}(t)=\int_{0}^{t}w_{i}(s)ds$
satisfy the $\Delta _{2}$-condition $W_{0}(2t)\leq (2+\frac{1}{w(1)}%
)W_{0}(t) $ and $W_{1}(2t)\leq 3W_{1}(t)$ for all $t>0$ (we skip the
detailed calculations). Thus, by observation in \cite[pp. 270--271]{KM04},
we have that $E^{(\ast )}=L^{1}(w_{0})^{(\ast )}=\Lambda _{1,w_{0}}$ and $%
F^{(\ast )}=L^{1}(w_{1})^{(\ast )}=\Lambda _{1,w_{1}}$ are quasi-Banach
spaces. Since $E+F=L^{1}(w)$ it follows that the space $(E+F)^{(\ast
)}=L^{1}(w)^{(\ast )}=\Lambda _{1,w}$, as $w$ is a decreasing function, is
even a Banach space. } }

\emph{{Note that for any $0<r<\infty $ the operator $H_{r}^{\ast }$ is not
bounded on $E=L^{1}(w_{0})$ neither on $F=L^{1}(w_{1})$. Namely, if $%
x_{0}=\chi _{\bigcup_{n=0}^{\infty }(2n,2n+1)}$, then $x_{0}\in E$ since 
\begin{equation*}
\int_{0}^{\infty }x_{0}(t)w_{0}(t)dt=\sum_{n=0}^{\infty
}\int_{2n}^{2n+1}w(t)dt\leq \int_{0}^{\infty }w(t)dt<\infty .
\end{equation*}%
However, 
\begin{eqnarray*}
H_{r}^{\ast }x_{0}(2n)^{r} &=&\int_{2n}^{\infty }\frac{|x_{0}(s)|^{r}}{s}%
ds=\sum_{k=n}^{\infty }\int_{2k}^{2k+1}\frac{|x_{0}(s)|^{r}}{s}ds \\
&=&\sum_{k=n}^{\infty }\int_{2k}^{2k+1}\frac{1}{s}ds\geq \sum_{k=n}^{\infty }%
\frac{1}{2k+1}=\infty
\end{eqnarray*}%
for any $n\in N$. Thus $H_{r}^{\ast }x_{0}(2n)=\infty $. Similarly, if $%
x_{1}=\chi _{\bigcup_{n=0}^{\infty }(2n+1,2n+2)}$, then $x_{1}\in F$ and $%
H_{r}^{\ast }x_{1}(2n+1)^{r}=\infty $ for any $n\in N$. } }

\emph{{Moreover, for any $r\geq 1$ the operator $H_{r}$ is not bounded on
the cone of nonnegative nonincreasing functions in $E$ and $F$. In fact, $%
x=x^{\ast }=\chi _{(0,1)}\in E\cap F=L^{1}(w_{0})\cap L^{1}(w_{1})$ but 
\begin{eqnarray*}
\Vert H_{r}x\Vert _{E} &\geq &\int_{1}^{\infty }\left( \frac{1}{t}%
\int_{0}^{t}|x(s)|^{r}ds\right) ^{1/r}w_{0}(t)dt=\int_{1}^{\infty
}t^{-1/r}w_{0}(t)dt \\
&\geq &\int_{\bigcup_{n=0}^{\infty
}(2n+1,2n+2)}t^{-1/r}w_{0}(t)dt=\sum_{n=0}^{\infty
}\int_{2n+1}^{2n+2}t^{-1/r}dt \\
&\geq &\sum_{n=0}^{\infty }\frac{1}{(2n+2)^{1/r}}=\infty .
\end{eqnarray*}%
Similarly, the operator $H_{r}$ is not bounded on $F^{\downarrow }$. } }
\end{example}

%%%%%%%%%%%%%

Let us consider now the problem of commutativity of the symmetrization
operation $E\mapsto E^{\left( \ast \right) }$ with the pointwise product and
the K\"{o}the duality. The first result is a simple consequence of the known
results and the second has been already concluded in \cite[Corollary 1.6]%
{KM07}, but we give a direct proof.

%%%%%%%%%%%%%%%%%%%%% Corollary 2

\begin{corollary}
\label{ilocz-gw-komut} Let $E$ and $F$ be quasi-Banach ideal spaces such
that $E^{(\ast )}\neq \{0\}$ and $F^{(\ast )}\neq \{0\}$. If the operator $%
D_{2}$ is bounded both on $E^{\downarrow }$ and $F^{\downarrow }$, then $%
E^{(\ast )}\odot F^{(\ast )}\neq \{0\},$ $E^{(\ast )}\odot F^{(\ast
)}\subset (E\odot F)^{(\ast )}$ and 
\begin{equation}
\left\Vert x\right\Vert _{(E\odot F)^{(\ast )}}\leq \left( C_{1}\right)
^{2}\left\Vert x\right\Vert _{E^{(\ast )}\odot F^{(\ast )}}~~\text{for all }%
x\in E^{(\ast )}\odot F^{(\ast )},  \label{inclusion8}
\end{equation}%
where $C_{1}$ is the constant from Theorem \ref{CLcommutstar} with the
function $\rho (s,t)=s^{1/2}t^{1/2}$. If, additionally, the operator $%
H_{r}^{\ast }$ is bounded on the spaces $E,F$ for some $r>0$, then $(E\odot
F)^{(\ast )}\subset E^{(\ast )}\odot F^{(\ast )}$ and%
\begin{equation}
\left\Vert x\right\Vert _{E^{(\ast )}\odot F^{(\ast )}}\leq \left(
C_{2}\right) ^{2}\left\Vert x\right\Vert _{(E\odot F)^{(\ast )}}\text{ }~~%
\text{for all }x\in (E\odot F)^{(\ast )},  \label{inlusion9}
\end{equation}%
where $C_{2}$ is the constant from Theorem \ref{CLcommutstar} with the
function $\rho (s,t)=s^{1/2}t^{1/2}.$ In particular, the inequalities (\ref%
{inclusion8}) and (\ref{inlusion9}) imply that the functional $\left\Vert
\cdot \right\Vert _{(E\odot F)^{(\ast )}}$ is a quasi-norm on the space $%
(E\odot F)^{(\ast )}$ and%
\begin{equation*}
(E\odot F)^{(\ast )}=E^{(\ast )}\odot F^{(\ast )}.
\end{equation*}
\end{corollary}

%%%%%%%%%%%%%%%%%%%

%TCIMACRO{\TeXButton{proof}{\proof}}%
%BeginExpansion
\proof%
%EndExpansion
Applying Theorem 1(iv) from \cite{KLM14}, Theorem \ref{CLcommutstar}(i),
commutativity of the $p$-convexifi\-cation with the symmetrization, that is,
the equality $(E^{(\ast )})^{(p)}\equiv (E^{(p)})^{(\ast )}$ and again
Theorem 1(iv) from \cite{KLM14} we get immediately%
\begin{eqnarray*}
\left\Vert x\right\Vert _{(E\odot F)^{(\ast )}} &=&\left\Vert x\right\Vert
_{[(E^{(1/2)}F^{(1/2)})^{(1/2)}]^{(\ast )}}=\left\Vert x\right\Vert
_{[(E^{(1/2)}F^{(1/2)})^{(\ast )}]^{(1/2)}} \\
&\leq &C_{1}^{2}\left\Vert x\right\Vert _{\left[ (E^{(\ast
)})^{(1/2)}(F^{(\ast )})^{(1/2)}\right] ^{(1/2)}}=C_{1}^{2}\left\Vert
x\right\Vert _{E^{(\ast )}\odot F^{(\ast )}}.
\end{eqnarray*}
This establishes inequality (\ref{inclusion8}) with the equality when
assumptions from Theorem \ref{CLcommutstar}(ii) are satisfied.%
%TCIMACRO{\TeXButton{proof}{\endproof}}%
%BeginExpansion
\endproof%
%EndExpansion

The commutativity of the symmetrization operation $E\mapsto E^{\left( \ast
\right) }$ with the K\"{o}the duality operation has been proved by Kami\'{n}%
ska and Masty{\l }o \cite[p. 231]{KM07}. We will give, however, a direct
proof and also show that the assumption on boundedness of $H^{\ast }$ on $E$
is essential.

%%%%%%%%%%%%%%%%%%%% Thm 2

\begin{theorem}
\label{dual-comut} Let $E$ be a quasi-Banach ideal space such that $E^{(\ast
)}\neq \{0\},$ the operator $D_{2}$ is bounded on $E^{\downarrow }$ and $%
(E^{\prime })^{(\ast )}\neq \{0\}$. Then $(E^{\prime })^{(\ast )}\subset
(E^{(\ast )})^{\prime }$ and 
\begin{equation}
\left\Vert x\right\Vert _{(E^{(\ast )})^{\prime }}\leq \left\Vert
x\right\Vert _{(E^{\prime })^{(\ast )}}~~\text{for all }x\in (E^{\prime
})^{(\ast )}.  \label{32}
\end{equation}%
If, additionally, the operator $H^{\ast }$ is bounded on the space $E$, then 
$(E^{(\ast )})^{\prime }\subset (E^{\prime })^{(\ast )}$ and 
\begin{equation}
\left\Vert x\right\Vert _{(E^{\prime })^{(\ast )}}\leq \left\Vert H^{\ast
}\right\Vert _{E\rightarrow E}\left\Vert x\right\Vert _{(E^{(\ast
)})^{\prime }}~~\text{for all }x\in (E^{(\ast )})^{\prime }.  \label{33}
\end{equation}%
In particular, the inequalities (\ref{32}) and (\ref{33}) imply that the
functional $\left\Vert \cdot \right\Vert _{(E^{\prime })^{(\ast )}}$ is a
quasi-norm on the space $(E^{\prime })^{(\ast )}$ and%
\begin{equation*}
(E^{\prime })^{(\ast )}=(E^{(\ast )})^{\prime }.
\end{equation*}
\end{theorem}

%%%%%%%%%%%%%%%%%%

\proof Clearly, $x\in (E^{\prime })^{(\ast )}$ if and only if $x^{\ast }\in
E^{\prime }$ and 
\begin{equation*}
\left\Vert x^{\ast }\right\Vert _{E^{\prime }}=\sup_{\left\Vert y\right\Vert
_{E}\leq 1}\int_{I}x^{\ast }(t)|y(t)|\,dt<\infty .
\end{equation*}%
Moreover, $x\in (E^{(\ast )})^{\prime }$ if and only if $x^{\ast }\in
(E^{(\ast )})^{\prime }$ and 
\begin{equation*}
\Vert x^{\ast }\Vert _{(E^{(\ast )})^{\prime }}=\sup_{\Vert y\Vert
_{E^{(\ast )}}\leq 1}\int_{I}x^{\ast }(t)|y(t)|\,dt=\sup_{\Vert y\Vert
_{E}\leq 1,\,y=y^{\ast }}\int_{I}x^{\ast }(t)|y(t)|\,dt<\infty .
\end{equation*}%
Thus, $(E^{\prime })^{(\ast )}\subset (E^{(\ast )})^{\prime }$ and
inequality (\ref{32}) is proved. To finish the proof we need to show the
inequality 
\begin{equation*}
\sup_{\Vert y\Vert _{E}\leq 1}\int_{I}x^{\ast }(t)|y(t)|\,dt\leq
C\sup_{\Vert y\Vert _{E}\leq 1,\,y=y^{\ast }}\int_{I}x^{\ast }(t)|y(t)|\,dt%
\text{ for }x\in (E^{(\ast )})^{\prime },
\end{equation*}%
where $C=\left\Vert H^{\ast }\right\Vert _{E\rightarrow E}$. Since $Hx^{\ast
}\geq x^{\ast }$ and using the duality of $H$ and $H^{\ast }$ we get 
\begin{eqnarray*}
\sup_{\Vert y\Vert _{E}\leq 1}\int_{I}x^{\ast }(t)|y(t)|\,dt &\leq
&\sup_{\Vert y\Vert _{E}\leq 1}\int_{I}(Hx^{\ast })(t)|y(t)|\,dt \\
&=&\sup_{\Vert y\Vert _{E}\leq 1}\int x^{\ast }(t)H^{\ast }(|y|)(t)\,dt \\
&\leq &\sup_{\Vert h\Vert _{E}\leq C,\,h=h^{\ast }}\int_{I}x^{\ast
}(t)h(t)\,dt \\
&\leq &C\sup_{\Vert y\Vert _{E}\leq 1,\,y=y^{\ast }}\int_{I}x^{\ast
}(t)y(t)\,dt.
\end{eqnarray*}%
We proved inequalities (\ref{32}) and (\ref{33}). This gives that the
functional $\left\Vert \cdot \right\Vert _{(E^{\prime })^{(\ast )}}$ is a
quasi-norm on the space $(E^{\prime })^{(\ast )}$ and $(E^{\prime })^{(\ast
)}=(E^{(\ast )})^{\prime }.$\endproof

%%%%%%%%%%%%%%%% Example 4

\begin{example}
\label{Ex4} \emph{{{We will give an example of a Banach ideal space $E$ on $%
(0, 1)$ such that $H^{\ast}$ is not bounded on $E$ and $(E^{\prime})^{(*)}
\not = (E^{(*)})^{\prime}$. }} }

\emph{{For $A_{n}=(2^{-n},2^{-n+1}),B_{n}=(2^{-n},3\times
2^{-n-1}),C_{n}=(3\times 2^{-n-1},2^{-n+1})$ and $a_{n}=(3/2)^{n}$ with $%
n=1,2,3,\ldots $ define two weights 
\begin{equation*}
w=\sum_{n=1}^{\infty }(\chi _{B_{n}}+a_{n}\chi
_{C_{n}})~~and~~v=\sum_{n=1}^{\infty }a_{n}\chi _{A_{n}}.
\end{equation*}%
Let $E=L^{1}(w)$ on $I=(0,1)$. We will show that $(E^{\prime })^{(\ast
)}\not=(E^{(\ast )})^{\prime }$ and $H^{\ast }$ is not bounded on $E$ . } }

\emph{{We have $E^{\prime} = L^{\infty}(1/w)$ and consequently $%
(E^{\prime})^{(*)} = L^{\infty}$. In fact, the incluson $L^{\infty} {%
\hookrightarrow } (E^{\prime})^{(*)}$ is evident, since $1/w\leq 1$. On the
other hand, if $x = x^*$ with $x(0^+) = \infty$, then choosing arbitrary $%
t_n \in B_n$ we have $x(t_n) \to \infty$, so that $x \not \in
L^{\infty}(1/w) $. } }

\emph{{On the other hand, it is easy to see that $E^{(*)} \not = L^1$
(further we will even identify this space). In fact, let $x =
\sum_{n=1}^{\infty} a_n \chi_{A_n}$. Then $x = x^*$ and $x \in L^1$.
Moreover, 
\begin{equation*}
\int_0^1x(t) w(t) \,dt \geq \int_{\cup_{n=1}^{\infty} C_n} x(t) w(t) \,dt =
\sum_{n=1}^{\infty} a_n^2 \,2^{-n-1} = 2^{-1}\sum_{n=1}^{\infty} (9/8)^{n}
=\infty,
\end{equation*}
and so $(E^{(\ast)})^{\prime} \not = L^{\infty}$, as claimed. } }

\emph{{Moreover, $D_2$ is bounded on $E$. In fact, we have $w(2t) \leq w(t)$
for $0\leq t\leq 1/2$, because if $t \in B_n$ then $2t \in B_{n-1}$ and if $%
t \in C_n$ then $2t\in C_{n-1}$. Therefore, 
\begin{equation*}
\| D_2x \|_E = \int_0^1x(t/2)w(t) \,dt = 2 \int_0^{1/2}x(t)w(2t) \,dt \leq
2\int_0^{1/2}x(t)w(t)dt \leq 2 \, \|x\|_E.
\end{equation*}
Now we identify $E^{(\ast)}$ by showing that $E^{(\ast)} = L^1(v)^{(\ast)} =
\Lambda_{1, v}$. } }

\emph{{In fact, the inclusion $L^{1}(v)^{(\ast )}\subset E^{(\ast )}$ is
obvious, since $v\geq w$ and so $L^{1}(v)\subset L^{1}(w)=E$. To see the
second inclusion, let $x=x^{\ast }\in L^{1}(w)$ and notice that $x(t)\leq
x(t/2)$ for each $t\in (0,1)$. Put $y(t)=x(t/2)$. Then also $y\in L^{1}(w)$,
since $D_{2}$ is bounded on the cone of nonincreasing nonnegative elements
of $E$. Moreover, we can see that 
\begin{equation*}
x(t)\leq y(t+2^{-n-1}),
\end{equation*}%
when $t\in B_{n}$, i.e., $t+2^{-n-1}\in C_{n}$ (we may say that $x\chi
_{B_{n}}$ is dominated by $y\chi _{C_{n}}$ shifted by $2^{-n-1}$ into the
right -- the best is to see it on the picture). Consequently, 
\begin{eqnarray*}
\int_{0}^{1}x(t)v(t)\,dt &=&\int_{\bigcup_{n=1}^{\infty
}B_{n}}x(t)v(t)\,dt+\int_{\bigcup_{n=1}^{\infty }C_{n}}x(t)v(t)\,dt \\
&=&\sum_{n=1}^{\infty }a_{n}\int_{B_{n}}x(t)\,dt+\int_{\bigcup_{n=1}^{\infty
}C_{n}}x(t)w(t)\,dt \\
&\leq &\sum_{n=1}^{\infty
}a_{n}\int_{C_{n}}y(t)\,dt+\int_{\bigcup_{n=1}^{\infty }C_{n}}y(t)w(t)\,dt \\
&=&2\int_{\bigcup_{n=1}^{\infty }C_{n}}y(t)w(t)\,dt\leq 2\,\Vert y\Vert
_{E}\leq 4\,\Vert x\Vert _{E}.
\end{eqnarray*}%
Then $x\in L^{1}(v)$ and so $E^{(\ast )}=L^{1}(w)^{(\ast )}=L^{1}(v)^{(\ast
)}=\Lambda _{1,v}$. } }

\emph{{At the end it may be instructive to see that the operator $H^{\ast }$
is not bounded on $E$. In order to prove this, we will show that the
operator $H$ is not bounded on $E^{\prime }=L^{\infty }(1/w)$. Evidently, $%
w\in L^{\infty }(1/w)$. For $2^{-n}<t<2^{-n+1}$ we have 
\begin{eqnarray*}
Hw(t) &=&\frac{1}{t}\int_{0}^{t}w(s)ds\geq
2^{n-1}\int_{0}^{2^{-n}}w(s)ds=2^{n-1}\sum_{k=n+1}^{\infty
}\int_{2^{-k}}^{2^{-k+1}}w(s)ds \\
&\geq &2^{n-1}\sum_{k=n+1}^{\infty }\int_{3\cdot
2^{-k-1}}^{2^{-k+1}}w(s)ds\geq 2^{n-1}\sum_{k=n+1}^{\infty }(\frac{3}{2}%
)^{k}2^{-k-1} \\
&=&2^{n-1}\frac{1}{2}\sum_{k=n+1}^{\infty }(\frac{3}{4})^{k}=2^{n-2}(\frac{3%
}{4})^{n+1}\cdot 4=\frac{1}{2}(\frac{3}{2})^{n+1},
\end{eqnarray*}%
and so 
\begin{eqnarray*}
\sup_{t\in (0,1]}\frac{Hw(t)}{w(t)} &=&\max_{n\in N}\sup_{2^{-n}<t<2^{-n+1}}%
\frac{Hw(t)}{w(t)} \\
&\geq &\frac{1}{2}\max_{n\in N}(3/2)^{n+1}\sup_{2^{-n}<t<2^{-n+1}}\frac{1}{%
w(t)}=\frac{1}{2}\max_{n\in N}(3/2)^{n+1}=\infty ,
\end{eqnarray*}%
that is, $Hw\not\in L^{\infty }(1/w).$ It means that the operator $H$ is not
bounded on $E^{\prime }=L^{\infty }(1/w)$. } }
\end{example}

%%%%%%%%%%%%%%%%%%%% Remark 3

\begin{remark}
\label{zero} If $E$ is a quasi-Banach ideal space such that $E^{(\ast
)}=\{0\}$, then 
\begin{equation*}
(E^{(\ast )})^{\prime }=\{x\in L^{0}\colon xy\in L^{1}~\text{for~all}~y\in
E^{(\ast )}\}=L^{0}
\end{equation*}%
and this is a trivial case as well as $(E^{\prime })^{(\ast )}=\{0\}$.
\end{remark}

%%%%%%%%%%%%%%%%

%%%%%%%%%%%%%%%% Section 4  %%%%%%%%%%%%%%%%%%%%%%%%%%

\section{Symmetrization of pointwise multipliers}

%%%%%%%%%%%%%%%%%%%%%%%%%%%

The next problem of our interest is the commutativity of the symmetrization
operation $E\mapsto E^{\left( \ast \right) }$ with the space of pointwise
multipliers. If $E,F$ are nontrivial symmetric spaces on $I$, then $M(E,F)$
is also symmetric space (see \cite[Theorem 2.2(i)]{KLM13}) and we obtain 
\begin{equation*}
M(E^{(\ast )},F^{(\ast )})=M(E,F)=M(E,F)^{(\ast )}.
\end{equation*}%
We want to have the same result for Banach ideal spaces. We will prove it
with the help of the \textquotedblleft arithmetic of function spaces", which
use our Theorem \ref{dual-comut} and some results from the paper \cite{KLM14}%
. Recall\ that $E^{(\ast )}$ is normable if there is a norm $\left\Vert
\cdot \right\Vert _{0}$ on $E^{(\ast )}$ which is equivalent to $\left\Vert
\cdot \right\Vert _{E^{(\ast )}}.$

%%%%%%%%%%%%%%%%%%%%%%%%% Thm 3

\begin{theorem}
\label{multip-komut} Let $E,F$ be Banach ideal spaces on $I$ such that $F$
has the Fatou property, $E^{(\ast )}\neq \{0\},F^{(\ast )}\neq \{0\}$ are
normable spaces, the operator $D_{2}$ is bounded on $F^{\downarrow },$ $%
(F^{\prime })^{(\ast )}\neq \{0\}$ and $\left( \left( E\odot F^{\prime
}\right) ^{\prime }\right) ^{\left( \ast \right) }\neq \{0\}$. Assume that
the following conditions hold:

\begin{itemize}
\item[$(i)$] The operator $H^{\ast }$ is bounded on the spaces $F$ and $%
E\odot F^{\prime }$.

\item[$(ii)$] For some $r>0,$ \ the operator $H_{r}^{\ast }$ is bounded on $%
E,F^{\prime }$, $\Vert H_{r}x^{\ast }\Vert _{E}\leq C_{E}\,\Vert x^{\ast
}\Vert _{E}$ for all $x^{\ast }\in E$ and $\Vert H_{r}x^{\ast }\Vert
_{F^{\prime }}\leq C_{F^{\prime }}\,\Vert x^{\ast }\Vert _{F^{\prime }}$ for
all $x^{\ast }\in F^{\prime }$.\end{itemize}
\noindent
Then 
\begin{equation*}
M(E^{(\ast )},F^{(\ast )})=M(E,F)^{(\ast )}.
\end{equation*}

\end{theorem}

%%%%%%%%%%%%%%%%%%%%%%%%

\vspace{-2mm}

\begin{proof} 
Applying in the subsequent steps several results from \cite{KLM14},  we obtain
\begin{eqnarray*}
M(E^{(*)}, F^{(*)}) &=&  M(E^{(*)}\odot (F^{(*)})^{\prime}, F^{(*)}\odot (F^{(*)})^{\prime}) \\
&& \hspace{1.5cm}  [{\rm by ~ Theorem ~ 4 ~ from ~ \cite{KLM14} ~ with} ~  G = F^{(*)} {\rm ~ a ~ Banach ~ space}] \\
&=& 
M(E^{(*)}\odot (F^{(*)})^{\prime}, L^{1}) \hspace{3mm} [{\rm by ~ the ~ Lozanovskii ~factorization ~ theorem}] \\
&=&
[E^{(*)}\odot (F^{(*)})^{\prime}]^{\prime} = [E^{(*)}\odot (F^{\prime})^{(*)}]^{\prime} \\
& &  \hspace{4cm} [{\rm by ~ Theorem ~ \ref{dual-comut}, ~ since} ~ H^{\ast} ~ {\rm is ~ bounded ~ on} ~F]  \\
&=&
[(E\odot F^{\prime})^{(*)}]^{\prime} \\
&&\hspace{1.5cm} [{\rm by ~ Corollary ~ \ref{ilocz-gw-komut},  ~using ~ the ~ assumptions ~ on} ~ H_r^{\ast} ~{\rm and} ~ H_r ] \\
&=&
[(E\odot F^{\prime})^{\prime}]^{(*)}  \hspace{2mm} [{\rm by ~ Theorem ~ \ref{dual-comut}, ~ since} ~ 
H^{\ast} ~{\rm is ~ bounded ~ on} ~E\odot F^{\prime}]\\
&=& 
M(E, F^{\prime \prime})^{(*)}   \hspace{3.5cm} [{\rm by ~ Corollary ~ 3 ~ from ~ \cite{KLM14}}] \\
&=&
M(E, F)^{(*)}   \hspace{3.5cm} [{\rm by ~ the ~ Fatou ~ property ~ of ~  F}].
\end{eqnarray*} 
\end{proof}

Note that the assumptions on $H_{r}^{\ast }$ and $H_{r}$ in $\left(
ii\right) $ allow us to apply Corollary \ref{ilocz-gw-komut} (see also
Remark \ref{Hr} and condition (\ref{dil})).

Let us comment assumptions of Theorem \ref{multip-komut}.\vspace{2mm}

\textrm{(a)} $(F^{\prime })^{(\ast )}\neq \{0\}$ and $H^{\ast }$ is bounded
on $F$ to get equality $(F^{(\ast )})^{\prime }=(F^{\prime })^{(\ast )}$ and 
$\left\Vert \cdot \right\Vert _{(F^{(\ast )})^{\prime }}\sim \left\Vert
\cdot \right\Vert _{(F^{\prime })^{(\ast )}}$ (see Theorem \ref{dual-comut}%
). In particular, since $(F^{(\ast )})^{\prime }\neq \{0\}$ is a Banach
space, the functional $\left\Vert \cdot \right\Vert _{(F^{\prime })^{(\ast
)}}$ is a quasi-norm.\newline
If we don't have boundedness of $H^{\ast }$ on $F$, then under assumption
that $(F^{\prime })^{(\ast )}\neq \{0\}$ we obtain only the inclusion $%
(F^{\prime })^{(\ast )}\overset{1}{\hookrightarrow }(F^{(\ast )})^{\prime }$
(see Example \ref{Ex4}), and in consequence only the inclusion $[E^{(\ast
)}\odot (F^{(\ast )})^{\prime }]^{\prime }\overset{1}{\hookrightarrow }%
[E^{(\ast )}\odot (F^{\prime })^{(\ast )}]^{\prime }$. \vspace{2mm}

\textrm{(b)} the operator $H_{r}^{\ast }$ is bounded on $E,F^{\prime }$ and
we have estimates of $H_{r}$ in $E,F^{\prime }$ on the cone of nonnegative
nonincreasing elements. \newline
If we don't have these assumptions but only $E^{(\ast )}\odot (F^{\prime
})^{(\ast )}\neq \{0\}$, then $E^{(\ast )}\odot (F^{\prime })^{(\ast )}%
\overset{C_{1}^{2}}{\hookrightarrow }(E\odot F^{\prime })^{(\ast )}$ with
the constant $C_{1}$ from Theorem \ref{CLcommutstar} (see also Corollary \ref%
{ilocz-gw-komut})), and so only inclusion $[(E\odot F^{\prime })^{(\ast
)}]^{\prime }\overset{1/C_{1}^{2}}{\hookrightarrow }[E^{(\ast )}\odot
(F^{\prime })^{(\ast )})]^{\prime }$ is valid. Finally, since the operator $%
H_{r}$ is bounded on $E^{\downarrow }$ ($\left( F^{\prime }\right)
^{\downarrow }$)$,$ so $E^{(\ast )}$ ($\left( F^{\prime }\right) ^{(\ast )}$%
) is a quasi-normed space by Remark \ref{Hr} and Corollary \ref%
{quasi-liniowosc}.\vspace{2mm}

\textrm{(c)} $H^{\ast }$ bounded on $E\odot F^{\prime }$. Without this
assumption we will have only the inclusion $[(E\odot F^{\prime })^{\prime
}]^{(\ast )}\overset{1}{\hookrightarrow }[(E\odot F^{\prime })^{(\ast
)}]^{\prime }$. \vspace{2mm}

Note also that for the spaces $E\odot F^{^{\prime }},(E\odot F^{^{\prime
}})^{^{\prime }}$ we need to have quasi-normed spaces after taking the
symmetrizations $(E\odot F^{\prime })^{(\ast )},[(E\odot F^{\prime
})^{\prime }]^{(\ast )}$. Notice that our assumptions imply that there are
constants $A,B>0$ such that%
\begin{equation}
A\left\Vert x\right\Vert _{(E\odot F^{\prime })^{(\ast )}}\leq \left\Vert
x\right\Vert _{E^{\left( \ast \right) }\odot \left( F^{\prime }\right)
^{\left( \ast \right) }}\leq B\left\Vert x\right\Vert _{(E\odot F^{\prime
})^{(\ast )}}  \label{rownow}
\end{equation}%
for all $x\in (E\odot F^{\prime })^{(\ast )}$ (see Corollary \ref%
{ilocz-gw-komut})$.$ Moreover, the functional $\left\Vert \cdot \right\Vert
_{E^{\left( \ast \right) }\odot \left( F^{\prime }\right) ^{\left( \ast
\right) }}$ is a quasi-norm because for quasi-normed spaces $X,Y$ the Calder%
\'{o}n space $X^{1/2}Y^{1/2}$ is quasi-normed and so is $X\odot Y=\left(
X^{1/2}Y^{1/2}\right) ^{\left( 1/2\right) }$ (cf. Corollary 1 from \cite%
{KLM14}). Consequently, by (\ref{rownow}), the functional $\left\Vert \cdot
\right\Vert _{(E\odot F^{\prime })^{(\ast )}}$ is also a quasi-norm.

The assumption $(\left( E\odot F^{\prime }\right) ^{\prime })^{\left( \ast
\right) }\neq \{0\}$ is necessary, because we apply Theorem \ref{dual-comut}
for the space $E\odot F^{\prime }.$ Finally, applying Theorem \ref%
{dual-comut} for the space $E\odot F^{\prime }$, we get $(\left( E\odot
F^{\prime }\right) ^{\prime })^{\left( \ast \right) }=(\left( E\odot
F^{\prime }\right) ^{\left( \ast \right) })^{\prime }$ with equivalent
respective functionals $\left\Vert \cdot \right\Vert _{(\left( E\odot
F^{\prime }\right) ^{\prime })^{\left( \ast \right) }}\sim \left\Vert \cdot
\right\Vert _{(\left( E\odot F^{\prime }\right) ^{\left( \ast \right)
})^{\prime }}$, whence in particular $(\left( E\odot F^{\prime }\right)
^{\left( \ast \right) })^{\prime }\neq \{0\}.$ Thus $(\left( E\odot
F^{\prime }\right) ^{\left( \ast \right) })^{\prime }$ is a Banach space.
Consequently, the functional on $(\left( E\odot F^{\prime }\right) ^{\prime
})^{\left( \ast \right) }$ is a quasi-norm.\vspace{2mm}

%%%%%%%%%%%%%%%%%%%% Remark 4

\begin{remark}
\label{R4} In Theorem \ref{multip-komut} we need to have that the spaces $%
E^{(\ast )}\neq \{0\},F^{(\ast )}\neq \{0\}$ are normable spaces, to be able
to use Theorem 4 from {\rm{\cite{KLM14}}}. Note that the condition $%
E^{(\ast )}\neq \{0\}$ has been discussed in Lemma \ref{nonzero}.
\end{remark}

%%%%%%%%%%%%%%%%

We are coming here to an interesting question.

%%%%%%%%%%%%%%%%%%%%% Problem 1

\begin{problem}
\label{P1} Characterize quasi-normed or normed ideal spaces $E$ for which $%
E^{(\ast )}$ is normable.
\end{problem}

%%%%%%%%%%%%%%%%%%%%%

It may happen that $E$ is a quasi-normed and $E^{(\ast )}$ is normable.
Indeed, if we take $E=L^{1}\left( 0,1/2\right) \oplus L^{1/2}[1/2,1)$ then $%
E^{\left( \ast \right) }=L^{1}\left( 0,1\right) .$

Moreover, if $E=L^{p}(w)$ with $0<p<\infty $ and with the weight $w\colon
(0,\infty )\rightarrow (0,\infty )$, then $E^{(\ast )}=(L^{p}(w))^{(\ast
)}=\Lambda _{p,w^{p}}$ is the \textit{Lorentz space} with the quasi-norm 
\begin{equation*}
\Vert x\Vert _{\Lambda _{p,w^{p}}}=\left( \int_{0}^{\infty }x^{\ast
}(t)^{p}w(t)^{p}\,dt\right) ^{1/p}.
\end{equation*}%
Assume that a weight function $w$ is locally integrable and satisfies the
following conditions:\newline
-- $W_{p}(t)=\int_{0}^{t}w(s)^{p}ds<\infty $ for all $t>0$ (this gives $%
\Lambda _{p,w^{p}}\neq \{0\}$), \newline
-- $W_{p}\in \Delta _{2}$, that is, there is a constant $C>0$ such that $%
W_{p}(2t)\leq CW_{p}(t)$ for all $t>0$ (then $\Lambda _{p,w^{p}}$ is a
linear space \cite[Remark 1.3 and Theorem 1.4]{CKMP}) and\newline
-- $W_{p}(\infty )=\int_{0}^{\infty }w(s)^{p}ds=\infty $, otherwise, $%
\Lambda _{p,w^{p}}$ is not separable.

Note that $\Lambda _{p,w^{p}}\neq \{0\}$ if and only if $%
\int_{0}^{t}w(s)^{p}ds<\infty $ for some $t>0.$ It is known that for $%
1<p<\infty $ the Lorentz space $\Lambda _{p,w^{p}}$ is normable if and only
if $\int_{t}^{\infty }s^{-p}w(s)^{p}ds\leq Ct^{-p}\int_{0}^{t}w(s)^{p}ds$
for all $t>0$ (see \cite[Theorem A]{KM04}, \cite[Theorem 12]{KMP07}).
Moreover, $\Lambda _{1,w}$ is normable if and only $W_{1}(t)/t$ is a
pseudo-decreasing function, that is, a decreasing with a constant (see \cite[%
Theorem 4]{KM04}, \cite[p. 104]{KMP07}). If $0<p<1$, then $\Lambda
_{p,w^{p}} $ is not normable since it contains copy of $l^{p}$ (see \cite[%
Theorem 1]{KM04}). \vspace{2mm}

If $E=L^{\infty }(w)$ with the weight $w\colon (0,\infty )\rightarrow
(0,\infty )$, then $E^{(\ast )}=(L^{\infty }(w))^{(\ast )}=M_{w}$ is the 
\textit{Marcinkiewicz space} generated by the functional $\Vert x\Vert
_{M_{w}}=\sup_{t>0}w(t)\,x^{\ast }(t)$. We do not exclude the case $%
M_{w}=\left\{ 0\right\} $. We assume that the fundamental function 
\begin{equation*}
\tilde{w}(t):=\Vert \chi _{(0,t)}\Vert _{M_{w}}=\sup_{s>0}w(s)\,\chi
_{(0,t)}(s)=\sup_{0<s<t}w(s)
\end{equation*}%
satisfies $\tilde{w}(t)<\infty $ for each $t>0.$ This function is increasing
and 
\begin{equation*}
\sup_{t>0}\tilde{w}(t)\,x^{\ast }(t)=\sup_{t>0}w(t)\,x^{\ast }(t).
\end{equation*}%
A nontrivial part of the proof of the last equality is the estimate 
\begin{equation*}
\sup_{t>0}\tilde{w}(t)\,x^{\ast }(t)=\sup_{t>0}[(\sup_{0<s\leq
t}w(s))\,x^{\ast }(t)]\leq \sup_{t>0}[\sup_{0<s\leq t}w(s)\,x^{\ast
}(s)]\leq \sup_{s>0}w(s)\,x^{\ast }(s).
\end{equation*}%
The functional $\Vert \cdot \Vert _{M_{w}}$ is a quasi-norm if and only if $%
\tilde{w}\in \Delta _{2}$, that is, there exists a constant $D\geq 1$ such
that $\tilde{w}(2t)\leq D\,\tilde{w}(t)$ for all $t>0$ (see Haaker \cite[%
Theorem 1.1]{Ha70}). The Marcinkiewicz space $M_{w}$ is normable if and only
if there exists a constant $B\geq 1$ such that $\int_{0}^{t}\dfrac{1}{\tilde{%
w}(s)}\,ds\leq B\dfrac{t}{\tilde{w}(t)}$ for all $t>0$ (see Haaker \cite[%
Theorem 2.4]{Ha70}).

%%%%%%%%%%%%%%%% Example 5

\begin{example}
\label{Examp5}

{\rm We will apply Theorem \ref{multip-komut} with $%
E=L^{p}(t^{a}),F=L^{q}(t^{b})$ and $a,b\in R$. We need to check the
respective assumptions.\newline
(a) Suppose {$1<q<p<\infty .$} Then: \newline
$1^{o}$\ $E^{(\ast )}\neq \{0\},F^{(\ast )}\neq \{0\}$ are normable spaces
if and only if $-1/p<a<1-1/p,-1/q<b<1-1/q$ -- see the discussion following
Problem \ref{P1}. Moreover, the operator $D_{2}$ is bounded on $%
F^{\downarrow }$ if and only if $W_{q}\in \Delta _{2}$ (see Corollary \ref%
{quasi-liniowosc} and \cite[Remark 1.3, Theorem 1.4]{CKMP}), which gives $%
b>-1/q.$\newline
$2^{o}$ $(F^{\prime })^{(\ast )}=[L^{q^{\prime }}(t^{-b})]^{(\ast )}\neq
\{0\}\Longleftrightarrow b<1-1/q;$\newline
$3^{o}$ the condition $(\left( E\odot F^{\prime }\right) ^{\prime })^{\left(
\ast \right) }=(L^{u}(t^{a-b})^{\prime })^{\left( \ast \right) }\neq \{0\},$ 
$1/u=1/p+1/q^{\prime },$ requires in particular that $\left(
L^{u}(t^{a-b})\right) ^{\prime }\neq \{0\}.$ Note that $u=\frac{pq}{pq+q-p}%
>1 $ because $1<q<p.$ Then $\{0\}\neq \left( L^{u}(t^{a-b})\right) ^{\prime
}=L^{u^{\prime }}(t^{b-a})$\emph{\ where }$\frac{1}{u}+\frac{1}{u^{\prime }}%
=1.$ Consequently, $\left( L^{u^{\prime }}(t^{b-a})\right) ^{\left( \ast
\right) }\neq \{0\}${\ }if and only if $\left( b-a\right) u^{\prime }>-1$.
Thus $\left( b-a\right) \frac{pq}{p-q}>-1$.\newline
Moreover, for the cases $4^{\circ }$--$7^{\circ }$ see section 1, the
discussion above inequality (\ref{equalityHardy}), \newline
$4^{o}$. $H^{\ast }$ is bounded on $F=L^{q}(t^{b})\Longleftrightarrow b>-1/q$%
, \newline
$5^{o}$. $H^{\ast }$ is bounded on $E\odot F^{\prime }=L^{p}(t^{a})\odot
L^{q^{\prime }}(t^{-b})=L^{u}(t^{a-b})$ if and only if: \newline
(i) $a-b>-1/u=-(1/p+1/q^{\prime })=1/q-1/p-1,$ this condition is satisfied
automatically by $1^{o},$\newline
(ii) $1\leq u=1/p+1/q^{\prime },$ whence $1/q-1/p\geq 0;$ this condition is
satisfied automatically by $q<p.$ \newline
$6^{o}$. $H_{r}^{\ast }$ is bounded on $E,F^{\prime }\Longleftrightarrow
a>-1/p,b<1-1/q$ (independent of $r$), \newline
$7^{o}$. $H_{r}$ is bounded on $E,F^{\prime }\Longleftrightarrow r\left(
a+1/p\right) <1{\ and}$ $r(-b+1/q^{\prime })<1$. For small $r>0$ the last
two estimates are valid, because $a+1/p\in \left( 0,1\right) $ and $b+1/q\in
\left( 0,1\right) $ by $1^{o}$. Summing up, the assumptions on $a,b$ are the
following 
\begin{equation}
-1/p<a<1-1/p,-1/q<b<1-1/q\text{ and }\left( b-a\right) \frac{pq}{p-q}>-1.
\label{Ex5, 10}
\end{equation}%
Theorem \ref{multip-komut} gives us that if $1<q<p<\infty $ and conditions (%
\ref{Ex5, 10}) hold, then 
\begin{equation}
M(\Lambda _{p,t^{ap}},\Lambda _{q,t^{bq}})=\left[ M\left(
L^{p}(t^{a}),L^{q}(t^{b})\right) \right] ^{(\ast )}=\left[ L^{s}(t^{b-a})%
\right] ^{(\ast )}=\Lambda _{s,t^{(b-a)s}},  \label{equality11}
\end{equation}%
where $1/s=1/q-1/p$. \newline
In the following cases we analogously check the required assumptions.\newline
(b) For $q=1<p<\infty .$ For the respective condition $3^{o}$ we have $%
(\left( E\odot F^{\prime }\right) ^{\prime })^{\left( \ast \right)
}=(L^{p^{\prime }}(t^{b-a}))^{\left( \ast \right) }\neq \left\{ 0\right\} $
holds if and only if $\left( b-a\right) \frac{p}{p-1}>-1.$ We check others
assumptions similarly getting the identification (\ref{equality11}) with $%
1/s=1-1/p$ under the assumption%
\begin{equation*}
-1/p<a<1-1/p,-1<b<0\text{ and }\left( b-a\right) \frac{p}{p-1}>-1.
\end{equation*}%
(c) For $q=1,p=\infty $ we have 
\begin{equation*}
M(M_{t^{a}},\Lambda _{1,t^{b}})=\left[ M\left( L^{\infty
}(t^{a}),L^{1}(t^{b})\right) \right] ^{(\ast )}=\left[ L^{1}(t^{b-a})\right]
^{(\ast )}=\Lambda _{1,t^{b-a}},
\end{equation*}%
whenever $-1<b\leq 0\leq a<1$ and $b-a>-1$ (the last inequality comes from
the assumption $3^{o}$). \newline
(d) For $1\leq p=q<\infty $ it holds 
\begin{equation*}
M(\Lambda _{p,t^{ap}},\Lambda _{p,t^{bp}})=\left[
M(L^{p}(t^{a}),L^{p}(t^{b}))\right] ^{(\ast )}=\left[ L^{\infty }(t^{b-a})%
\right] ^{(\ast )}=M_{t^{b-a}},
\end{equation*}%
whenever $-1/p<a,b<1-1/p$ and $0\leq b-a<1/p$ (the last inequality comes
from the assumption $3^{o}$ and $5^{o}$). \newline
(e) For $p=q=\infty $ it holds 
\begin{equation*}
M(M_{t^{a}},M_{t^{b}})=\left[ M(L^{\infty }(t^{a}),L^{\infty }(t^{b}))\right]
^{(\ast )}=\left[ L^{\infty }(t^{b-a})\right] ^{(\ast )}=M_{t^{b-a}},
\end{equation*}%
whenever $a,b\in \left( 0,1\right) $ and $b-a\geq 0$ (the last inequality
comes from the assumption $3^{o})$. Note that for $r\in \left( 0,\min (1,%
\frac{1}{a},\frac{1}{1+b}\right) )$ the assumption $\left( ii\right) $ of
Theorem \ref{multip-komut} is satisfied. Note also that $M_{t^{b-a}}=\left\{
0\right\} $ if $b<a.$ }
\end{example}

The above result can be applied to describe the space of multipliers between
classical Lorentz spaces $L^{p,q}$, which particular cases were proved in 
\cite{Na95}--\cite{Na17}. \vspace{2mm}

For $0<p,q\leq \infty $ consider the classical \textit{Lorentz function
spaces $L^{p,q}=L^{p,q}(I)$} on $I=(0,1)$ or $I=(0,\infty )$ defined by the
quasi-norms 
\begin{equation*}
\Vert x\Vert _{p,q}=\left\{ 
\begin{array}{ll}
\mbox{$\Big(\int\limits_0^{m(I)} [t^{1/p} x^*(t)]^q \frac{dt}{t}
\Big)^{1/q}$,} & \text{for }0<p\leq \infty ,0<q<\infty ,\vspace{1mm} \\ 
\mbox{$\sup\limits_{0 < t < m(I)} t^{1/p} x^*(t)$,} & 
\mbox{for $0 < p \leq
\infty, q = \infty$}.\vspace{1mm}%
\end{array}%
\right.
\end{equation*}%
Note that $L^{p,p}\equiv L^{p}$ for $0<p\leq \infty $ and $L^{\infty
,q}=\{0\}$ for $0<q<\infty .$ Consequently, for $p=\infty $ we consider only
the case $q=\infty .$ Recall that $(L^{p,q})^{(r)}=L^{pr,qr},$ where $%
(L^{p,q})^{(r)}$ is the $r$--convexification ($1<r<\infty $) of the Lorentz
space $L^{p,q}$.

We characterize below all multipliers $M(L^{p_{1},q_{1}},L^{p_{2},q_{2}}).$
First, we describe cases when one of spaces $L^{p_{1},q_{1}}$ or $%
L^{p_{2},q_{2}}$ is equal to $L^{\infty }$ because this limit cases do not
suit to the formal model of the below theorem.

\begin{remark}
If $p_{1}=q_{1}=\infty ,$ then $%
M(L^{p_{1},q_{1}},L^{p_{2},q_{2}})=L^{p_{2},q_{2}}$ for all $p_{2},q_{2}>0.$
If $p_{2}=q_{2}=\infty $ we need only to consider the case $0<p_{1}<\infty $
in which $M(L^{p_{1},q_{1}},L^{p_{2},q_{2}})=M(L^{p_{1},q_{1}},L^{\infty
})=\left\{ 0\right\} $ for each $q_{1}>0$ by Proposition 2.3(ii),(iv) in 
\rm{\cite{KLM13}}.
\end{remark}

\begin{theorem}
\label{Lorentz-multi} Let $0<p_{1},p_{2}<\infty ,$ $0<q_{1},q_{2}\leq \infty 
$ and $I=\left( 0,1\right) $ or $I=\left( 0,\infty \right) $.

\begin{itemize}
\item[$(i)$] If either $p_{1}<p_{2}$ or $p_{1}=p_{2}$ and $q_{1}>q_{2},$
then $M(L^{p_{1},q_{1}},L^{p_{2},q_{2}})=\{0\}.$

\item[$(ii)$] If either $p_{1}>p_{2}$ or $p_{1}=p_{2}$ and $q_{1}\leq q_{2},$
then 
\begin{equation*}
M(L^{p_{1},q_{1}},L^{p_{2},q_{2}})=L^{p_{3},q_{3}},
\end{equation*}%
where%
\begin{equation*}
\frac{1}{p_{3}}=\frac{1}{p_{2}}-\frac{1}{p_{1}}\text{ and }\frac{1}{q_{3}}%
=\left\{ 
\begin{array}{ccc}
\frac{1}{q_{2}}-\frac{1}{q_{1}} & \text{if} & q_{1}>q_{2}, \\ 
0 & \text{if} & q_{1}\leq q_{2}%
\end{array}%
\right. .
\end{equation*}
\end{itemize}
\end{theorem}

\begin{proof}
$\left( i\right) $ It is enough to apply Proposition 2.3(ii),(iv)
in \cite{KLM13} and the respective\ strict embeddings between Lorentz spaces 
$L^{p,q}$ which are well known$.$\ \newline
We present, however, also a direct proof which idea has been taken from \cite[%
Theorem 2]{MP89} (see also \cite[Proposition 2.3]{KLM13}). We have two
inclusions: $L^{p,q_{1}}\hookrightarrow L^{p,q_{2}}$ for $0<q_{1}<q_{2}\leq
\infty $ and for any $0<p<\infty $ (see \cite[Proposition 4.2]{BS88} or \cite%
[Proposition 1.4.10]{Gr08}); also if $0<m(A)<\infty $, then $L^{p_{2},\infty
}(A)\hookrightarrow L^{p_{1},q}(A)$ for $0<p_{1}<p_{2}<\infty $ and for any $%
0<q\leq \infty $ since for $0<q<\infty $ 
\begin{eqnarray*}
\Vert x\Vert _{p_{1},q} &=&(\int_{0}^{m(A)}[t^{1/p_{1}}x^{\ast }(t)]^{q}%
\frac{dt}{t})^{1/q} \\
&\leq &\sup\limits_{0<t<m(A)}t^{1/p_{2}}x^{\ast
}(t)\,(\int_{0}^{m(A)}t^{(1/p_{1}-1/p_{2})q}\,\frac{dt}{t})^{1/q}=C\,\Vert
x\Vert _{p_{2},\infty },
\end{eqnarray*}%
where $C=C(p_{1},p_{2},q,m(A))$. If $q=\infty $, then 
\begin{equation*}
\Vert x\Vert _{p_{1},\infty }\leq \Vert x\Vert _{p_{2},\infty
}\sup\limits_{0<t<m(A)}t^{1/p_{1}-1/p_{2}}=m(A)^{1/p_{1}-1/p_{2}}\,\Vert
x\Vert _{p_{2},\infty }.
\end{equation*}

Now, if there exists $0 \neq x \in M(L^{p_1, q_1}, L^{p_2, q_2})$, then $%
|x(t)| > 0$ for almost all $t \in A$ with $0 < m(A) < \infty$. Let $A_n =
\{t \in A \colon \frac{1}{n} \leq |x(t)| \leq n\}, n = 1, 2, \ldots$. Then $%
A_n \nearrow A$ and so $m(A_n) > 0$ for $n \geq n_0$. If $y \in L^{p_1,
q_1}(A_n)$, then $y\chi_{A_n} \in L^{p_1, q_1}(I)$ and 
\begin{equation*}
\frac{1}{n} y \chi_{A_n} \leq x y\chi_{A_n} \in L^{p_2, q_2}(A_n).
\end{equation*}
Hence, $L^{p_1, q_1}(A_n) \hookrightarrow L^{p_2, q_2}(A_n)$ for $n \geq n_0$%
. This is, of course, not possible for $p_1 < p_2$ since then it will be $%
L^{p_2, q_2}(A_n) \hookrightarrow L^{p_2, \infty}(A_n) \hookrightarrow
L^{p_1, q_1}(A_n)$ and consequently $L^{p_1, q_1}(A_n) = L^{p_2, q_2}(A_n)$,
which is not possible because $m(A_n) > 0$ for $n \geq n_0$.

Also it is not possible for $p_{1}=p_{2}$ and $q_{2}<q_{1}$ since again $%
L^{p_{2},q_{2}}(A_{n})\hookrightarrow L^{p_{1},q_{2}}(A_{n})\hookrightarrow
L^{p_{1},q_{1}}(A_{n})$ and we get a contradiction.

$\left( ii\right) $ Assume first that $1<p_{1},p_{2},q_{1},q_{2}$. Note that 
$L^{p,q}$ is normable for $p,q>1$. We divide the proof into three parts.%
\newline

Case A. Suppose that $1<p_{2}<p_{1}<\infty $ and $1<q_{2}\leq q_{1}\leq
\infty $. It means that $1/p_{3}=1/p_{2}-1/p_{1}>0$ and $%
1/q_{3}=1/q_{2}-1/q_{1}\geq 0$. Thus we can write 
\begin{equation*}
L^{p_{3},q_{3}}\odot L^{p_{1},q_{1}}=L^{p_{2},q_{2}},
\end{equation*}%
according to \cite{CS17} (cf. \cite{KLM14}). Applying cancellation law from 
\cite[Theorem 4]{KLM14} we see that 
\begin{equation*}
M(L^{p_{1},q_{1}},L^{p_{2},q_{2}})=M(L^{p_{1},q_{1}}\odot L^{\infty
},L^{p_{3},q_{3}}\odot L^{p_{1},q_{1}})=M(L^{\infty
},L^{p_{3},q_{3}})=L^{p_{3},q_{3}}.
\end{equation*}

We also present an independent proof basing on Example \ref{Examp5}. Let $%
1<p_{2}<p_{1}$ and $1<q_{2}<q_{1}<\infty $. Then, from Example \ref{Examp5}%
(a) and (b) with $a=1/p_{1}-1/q_{1},b=1/p_{2}-1/q_{2}$, we obtain 
\begin{equation*}
M(L^{p_{1},q_{1}},L^{p_{2},q_{2}})=M(\Lambda
_{q_{1},t^{q_{1}/p_{1}-1}},\Lambda _{q_{2},t^{q_{2}/p_{2}-1}})
\end{equation*}%
\begin{equation*}
=\Lambda _{q_{3},t^{(b-a)q_{3}}}=\Lambda
_{q_{3},t^{q_{3}/p_{3}-1}}=L^{p_{3},q_{3}}.
\end{equation*}

Case B. Let $1<p_{2}<p_{1}<\infty $ and $1<q_{1}<q_{2}\leq \infty $. Then it
is easy to see that assumptions of Theorem 2.2 in \cite{KLM13} are
satisfied. Consequently, the fundamental function of $%
M(L^{p_{1},q_{1}},L^{p_{2},q_{2}})$ is equivalent to $t^{1/p_{2}-1/p_{1}}$.
Applying the maximality of Marcinkiewicz space (see \cite{BS88} and
condition (29) in \cite{KLM14}) w get 
\begin{equation*}
M(L^{p_{1},q_{1}},L^{p_{2},q_{2}})\subset
M_{t^{1/p_{2}-1/p_{1}}}=L^{p_{3},\infty }.
\end{equation*}%
On the other hand, using once again \cite{CS17} we can write 
\begin{equation*}
L^{p_{3},\infty }\odot L^{p_{1},q_{1}}=L^{p_{2},q_{1}}\subset
L^{p_{2},q_{2}}.
\end{equation*}%
Thus 
\begin{equation*}
L^{p_{3},\infty }\subset M(L^{p_{1},q_{1}},L^{p_{2},q_{1}})\subset
M(L^{p_{1},q_{1}},L^{p_{2},q_{2}}).
\end{equation*}%

Case C. Suppose $1<p_{1}=p_{2}<\infty $ and $1<q_{1}\leq q_{2}$. Note that
both spaces $L^{p_{2},1}$ and $L^{p_{2},\infty }$ have the same fundamental
functions. Moreover, $L^{p_{2},1}\subset L^{p_{2},q_{1}}$ and $%
L^{p_{2},q_{2}}\subset L^{p_{2},\infty }$ (see the proof of $\left( i\right) 
$)$.$ In consequence, 
\begin{equation*}
L^{\infty }\subset M(L^{p_{2},q_{1}},L^{p_{2},q_{2}})\subset
M(L^{p_{2},1},L^{p_{2},\infty })\subset L^{\infty },
\end{equation*}%
where the last inclusion follows from \cite{KLM14}, Proposition 3. In fact,
we have the equality $M(L^{p_{2},1},L^{p_{2},\infty })=L^{\infty },$ becuase
the reverse embedding $L^{\infty }\subset M(L^{p_{2},1},L^{p_{2},\infty })$
is clear by $L^{p_{2},1}\subset L^{p_{2},\infty }.$\newline
In order to remove assumption $1<p_{2},p_{1},q_{2},q_{1}$ we need to notice
that 
\begin{equation*}
(L^{p_{2},q_{2}})^{(r)}=L^{p_{2}r,q_{2}r}
\end{equation*}%
and 
\begin{equation*}
M(E^{(r)},F^{(r)})=M(E,F)^{(r)},
\end{equation*}%
for arbitrary $r>0$. Since multiplying by $r$ does not damage relations
between $p_{2},p_{1},q_{2}$ and $q_{1}$, we may choose $r$ so that $%
1<rp_{2},rp_{1},rq_{2},rq_{1}$ and apply the previous part of theorem
together with above equalities. 
\end{proof}

In particular, for $p_{2}=q_{2}=1$ we obtain from Theorem \ref{Lorentz-multi}
the well-known results on K\"{o}the duality (see \cite{Cw73}, \cite{Cw75}, 
\cite{CS72}, \cite{Hu66} and also Grafakos book \cite[pp. 52--55]{Gr08}):

(a) If $0 < p < 1$ and $0 < q < \infty$, then $(L^{p, q})^{\prime} = \{0\}$
(Hunt \cite[p. 262]{Hu66}).

(b) If $0 < p < 1$, then $(L^{p, \infty})^{\prime} = \{0\}$ (Haaker \cite[%
Theorem 4.2]{Ha70}, Cwikel \cite[Theorem 1]{Cw73}).

(c) If $1<q\leq \infty $, then $(L^{1,q})^{\prime }=\{0\}$ (Hunt \cite[pp.
262--263]{Hu66}; observe that Cwikel--Sagher \cite{CS72} proved that the
dual space of $(L^{1,\infty })^{\ast }\neq \{0\}$ but there is no its exact
description).

(d) If $1<p<\infty ,1\leq q\leq \infty $, then $(L^{p,q})^{\prime
}=L^{p^{\prime },q^{\prime }}$ (Hunt \cite[Theorem 2.7]{Hu66}).

(e) If $1<p<\infty ,0<q\leq 1$, then $(L^{p,q})^{\prime }=L^{p^{\prime
},\infty }$ (Hunt \cite[Theorem 2.7]{Hu66}). \vspace{1mm}

Moreover, for example, for $q_{1}=p_{2}=1$ we obtain from Theorem \ref%
{Lorentz-multi} some probably new results on pointwise multipliers:

(f) If $0 < p < 1$ and $0 < q < \infty$, then $M(L^{p, 1}, L^{1, q}) = \{0\}$%
.

(g) If $p=1$ and $1\leq q\leq \infty $, then $M(L^{1},L^{1,q})=L^{\infty }$.
Moreover, $M(L^{1},L^{1,q})=\left\{ 0\right\} $ if $0<q<1.$

(h) If $1<p<\infty $ and $1\leq q<\infty $, then $M(L^{p,1},L^{1,q})=L^{p^{%
\prime },\infty }$.

%%%%%%%%%%%%%%%%%%%%%%%%%%%%%%%%%%%%% Chapter 4

\section{ Explicit factorization of product spaces}

A {factorization of function or sequence spaces is a powerful tool which
found applications in interpolation theory, geometry of Banach spaces (for
example the idea of indicator function from \cite{Ka92, Ka95}) and operator
theory (for example the proof of Nehari theorem in \cite{Pel03}). Usually it
is enough to know that for each $f\in G$ there exist $g\in E,h\in F$
satisfying $f=gh$ with $\Vert f\Vert _{G}\approx \Vert g\Vert _{E}\Vert
h\Vert _{F}$, i.e. $G=E\odot F$. However, in some cases (see for example 
\cite{CS17}), the existence is not enough and one prefers to know explicite
formulas which for a given $f$ produce $g$ and $h$ as above. }

It is evident how to factorize $f\in L^{p}$ in order to get $f=gh$
satisfying $g\in L^{q},h\in L^{r}$ and $\Vert f\Vert _{p}=\Vert g\Vert
_{q}\Vert h\Vert _{r}$, i.e. $g=|f|^{p/q}$ and $h=|f|^{p/r}\mathrm{sgn}f$,
where $1/p=1/q+1/r$ (see Example \ref{Ex5} below). Similar explicite
formulas follow directly from the respective factorization theorems for
Orlicz spaces \cite[Theorem 10.1(b)]{Ma89} and Calder\'{o}n--Lozanovski\u{\i}
spaces \cite[Theorem 5]{KLM14}, or for Lorentz spaces \cite{CS17}. In this
section we explain how to derive expilicit formulas for factorization of
symmetrized space, once we know the respective formulas for initial space.

\begin{definition}
\label{defexplicite} \textrm{Let $G=E\odot F$. We will say that the \textit{%
explicit factorization for $G=E\odot F$ holds} if we have explicit formulas
for maps $\varphi \colon G\rightarrow E$ and $\psi \colon G\rightarrow F$
such that each $x\in G$ can be written as 
\begin{equation*}
x=\varphi (x)\psi (x)~~and~~\Vert x\Vert _{G}\approx \Vert \varphi (x)\Vert
_{E}\Vert \psi (x)\Vert _{F}.
\end{equation*}%
}
\end{definition}

\begin{example}
\label{Ex5}

\rm{ Supppose $E$ and $F$ are quasi--Banach ideal spaces and $w,w_{0},w
_{1}$ are positive weights such that $w=w_{0}\,w_{1}$. If $G=E\odot F$ and
the explicit factorization holds, then $G(w)=E(w_{0})\odot F(w_{1})$ and the
explicit factorization holds. }

\rm { In particular, $L^{p}(w)=L^{p_{0}}(w_{0})\odot L^{p_{1}}(w_{1})$
with $1/p=1/p_{0}+1/p_{1},w=w_{0}\,w_{1}$ and the maps 
\begin{equation*}
\varphi (x)=\left[ (|x|w)^{p/p_{0}}\mathrm{sgn}x\right] /w_{0}~~and~~\psi
(x)=(|x|w)^{p/p_{1}}/w_{1},
\end{equation*}%
satisfy conditions of Definition \ref{defexplicite}. }
\end{example}

%TCIMACRO{\TeXButton{proof}{\begin{proof}}}%
%BeginExpansion
\begin{proof}%
%EndExpansion
We will show the respective equalities and also maps, which give that
explicit factorizations. \vspace{2mm}

Firstly, we show that $G=E\odot F$ implies $G(w)=E(w_{0})\odot F(w_{1})$
with $w=w_{0}\,w_{1}$. In fact, if $x\in E(w_{0})$ and $y\in F(w_{1})$, then 
$xw_{0}\in E,yw_{1}\in F$, and by assumption $xw_{0}yw_{1}\in G$, whence $%
xy\in G(w)$, that is, $E(w_{0})\odot F(w_{1})\hookrightarrow G(w)$.

Conversely, if $x\in G(w)$, then $x\,w\in G$ and, by the assumption that for 
$G$ the explicit factorization holds, there are maps $\varphi \colon
G\rightarrow E$ and $\psi \colon G\rightarrow F$ such that 
\begin{equation*}
x\,w=\varphi (x\,w)\,\psi (x\,w)~~{and}~~\Vert x\,w\Vert _{G}\approx \Vert
\varphi (x\,w)\Vert _{E}\Vert \psi (x\,w)\Vert _{F}.
\end{equation*}%
Taking 
\begin{equation}
\varphi _{w}(x)=\varphi (x\,w)/w_{0}~~\mathrm{and}~~\psi _{w}(x)=\psi
(x\,w)/w_{1}  \label{explicit2}
\end{equation}%
we obtain $\varphi _{w}(x)\psi _{w}(x)=\frac{\varphi (x\,w)}{w_{0}}\frac{%
\psi (x\,w)}{w_{1}}=\frac{x\,w}{w}=x$ and 
\begin{equation*}
\Vert \varphi _{w}(x)\Vert _{E(w_{0})}\Vert \psi _{w}(x)\Vert
_{F(w_{1})}=\Vert \varphi (x\,w)\Vert _{E}\Vert \psi (x\,w)\Vert _{F}\approx
\Vert x\,w\Vert _{G}=\Vert x\Vert _{G(w)}.
\end{equation*}%
Therefore, $G(w)\hookrightarrow E(w_{0})\odot F(w_{1})$ and (\ref{explicit2}%
) is the explicit factorization.%
%TCIMACRO{\TeXButton{proof}{\end{proof}}}%
%BeginExpansion
\end{proof}%
%EndExpansion

%%%%%%%%%%%%%%%%%%%% Theorem 5

\begin{theorem}
\label{Th-explicit}Supppose $E$ and $F$ are quasi-Banach ideal spaces such
that the operators $H_{r}$ and $H_{r}^{\ast }$ are bounded on $%
E^{(1/2)},F^{(1/2)}$ for some $r>0$. If for $G=E\odot F$ the explicit
factorization holds, then $G^{(\ast )}=E^{(\ast )}\odot F^{(\ast )}$ and the
explicit factorization holds.
\end{theorem}

%%%%%%%%%%%%%%%%%%%%

\begin{remark}
It is easy to prove that the following conditions are equivalent: $\left(
a\right) $ the operator $H_{r}$ is bounded on $E^{(1/2)},$ $\left( b\right) $
the operator $H$ is bounded on $E^{(1/2r)},$ $\left( c\right) $ the operator 
$H_{2r}$ is bounded on $E.$
\end{remark}

\begin{proof}[Proof of Theorem \ref{Th-explicit}]  Note that, by the
assumption on the operator $H_{r}$ and Remark \ref{Hr}, the condition (\ref%
{dil}) is satisfied for the spaces $E,F$. First, we will prove the inclusion 
$E^{(\ast )}\odot F^{(\ast )}\overset{C}{\hookrightarrow }G^{(\ast )}$,
where $C=A_{E}A_{F},$ where $A_{E},A_{F}$ are the best constants in (\ref%
{dil}). Suppose $x\in E^{(\ast )},y\in F^{(\ast )}$. Then $x^{\ast }\in
E,y^{\ast }\in F$ and by the assumption, $u=x^{\ast }\,y^{\ast }\in G$. We
need to show that $x\,y\in G^{(\ast )}$ or equivalently $(x\,y)^{\ast }\in G$%
. But 
\begin{equation*}
(x\,y)^{\ast }(t)\leq x^{\ast }(t/2)\,y^{\ast }(t/2)=D_{2}x^{\ast
}(t)D_{2}y^{\ast }(t).
\end{equation*}%
Moreover, $D_{2}x^{\ast }\in E,\,D_{2}y^{\ast }\in F$, by the assumption
which finishes the proof of the first inclusion.

Now, we want to prove the reverse inclusion $G^{(\ast )}\hookrightarrow
E^{(\ast )}\odot F^{(\ast )}$. Let $x\in G^{(\ast )}$. Then $x^{\ast }\in G$
and, by the assumption that the explicit factorization holds, 
\begin{equation*}
x^{\ast }=\varphi (x^{\ast })\,\psi (x^{\ast }),\varphi (x^{\ast })\in
E,\psi (x^{\ast })\in F\text{ with }\Vert x^{\ast }\Vert _{G}\approx \Vert
\varphi (x^{\ast })\Vert _{E}\Vert \psi (x^{\ast })\Vert _{F}.
\end{equation*}
For $\varphi _{1}(x)=\varphi (x^{\ast })^{2}$ and $\psi _{1}(x)=\psi
(x^{\ast })^{2}$ we have 
\begin{equation}
x^{\ast }=\varphi _{1}(x)^{1/2}\,\psi _{1}(x)^{1/2}  \label{rozklad-x*}
\end{equation}%
and, applying twice the H\"{o}lder--Rogers inequality and the equality (\ref%
{equalityHardy}), we get 
\begin{eqnarray*}
x^{\ast } &\leq &H_{r}(x^{\ast })\leq \left[ H_{r}(x^{\ast
})^{r}+H_{r}^{\ast }(x^{\ast })^{r}\right] ^{1/r}=H_{r}H_{r}^{\ast }(x^{\ast
})=H_{r}H_{r}^{\ast }\left( \varphi _{1}(x)^{1/2}\,\psi _{1}(x)^{1/2}\right)
\\
&\leq &H_{r}\{[H_{r}^{\ast }(\varphi _{1}(x))]^{1/2}[H_{r}^{\ast }(\psi
_{1}(x))]^{1/2}\}\leq \lbrack H_{r}H_{r}^{\ast }(\varphi
_{1}(x))]^{1/2}[H_{r}H_{r}^{\ast }(\psi _{1}(x))]^{1/2}.
\end{eqnarray*}%
Define 
\begin{equation*}
\varphi _{2}(x)=[H_{r}\,H_{r}^{\ast }(\varphi _{1}(x))]^{1/2}\ \mathrm{and}\
\psi _{2}(x)=\frac{x^{\ast }}{\varphi _{2}(x)}.
\end{equation*}%
Clearly, $x^{\ast }=\varphi _{2}(x)\,\psi _{2}(x)$. Denote $C_{0}=\left\Vert
H_{r}\right\Vert _{E^{(1/2)}\rightarrow E^{(1/2)}}$, $C_{1}=\left\Vert
H_{r}\right\Vert _{F^{(1/2)}\rightarrow F^{(1/2)}}$, $D_{0}=\left\Vert
H_{r}^{\ast }\right\Vert _{E^{(1/2)}\rightarrow E^{(1/2)}}$ and $%
D_{1}=\left\Vert H_{r}^{\ast }\right\Vert _{F^{(1/2)}\rightarrow F^{(1/2)}}$%
. Note that $\varphi _{1}(x)\in E^{(1/2)},\psi _{1}(x)\in F^{(1/2)}$. Then $%
\varphi _{2}(x)\in E$ because 
\begin{eqnarray*}
\left\Vert \varphi _{2}(x)\right\Vert _{E} &=&\left\Vert [H_{r}H_{r}^{\ast
}(\varphi _{1}(x))]^{1/2}\right\Vert _{E}=\left\Vert H_{r}H_{r}^{\ast
}(\varphi _{1}(x))\right\Vert _{E^{(1/2)}}^{1/2} \\
&\leq &C_{0}^{1/2}\,D_{0}^{1/2}\left\Vert \varphi _{1}(x)\right\Vert
_{E^{(1/2)}}^{1/2}=C_{0}^{1/2}\,D_{0}^{1/2}\left\Vert \varphi (x^{\ast
})\right\Vert _{E}.
\end{eqnarray*}%
Since similarly we have $\psi _{2}(x)\leq \lbrack H_{r}H_{r}^{\ast }(\psi
_{1}(x))]^{1/2}\in F$ so $x^{\ast }\in E\odot F$.

\noindent Recall that the rank function $r_{x}$ of the function $x$ is
defined by the formula 
\begin{equation*}
r_{x}(t)=m \{ s:\left|x(s)\right| > \left|x(t)\right| \ or \
\left|x(s)\right| = \left|x(t)\right| \ \mathrm{and} \ s \leq t \}.
\end{equation*}
It is known that $r_{x}$ preserves measure and $\left|x\right|=x^{*}\circ
r_{x} $ a.e. provided $x^{*}(\infty)=0$ (see Proposition 2 and 3 in \cite%
{Ryff70}).

Since $\varphi _{2}(x)$ is decreasing it follows that $(\varphi _{2}(x)\circ
(r_{x}))^{\ast }=(\varphi _{2}(x))^{\ast }=\varphi _{2}(x)\in E$, whence $%
\varphi _{2}(x)\circ (r_{x})\in E^{(\ast )}$. Similarly, we have $\psi
_{2}(x)\circ (r_{x})\leq \lbrack H_{r}H_{r}^{\ast }(\psi
_{1}(x))]^{1/2}\circ (r_{x})\in F^{(\ast )}$. Thus, $x\in E^{(\ast )}\odot
F^{(\ast )}$ and so $G^{(\ast )}\hookrightarrow E^{(\ast )}\odot F^{(\ast )}$%
. Moreover, the explicit factorization for $|x|$ is given by the formula 
\begin{equation*}
\varphi _{3}(x)=[H_{r}H_{r}^{\ast }(\varphi _{1}(x))]^{1/2}\circ (r_{x})~~%
\mathrm{and}~~\psi _{3}(x)=\frac{|x|}{[H_{r}H_{r}^{\ast }(\varphi
_{1}(x))]^{1/2}\circ (r_{x})},
\end{equation*}%
where $x^{\ast }=\varphi (x^{\ast })\,\psi (x^{\ast })$ is the respective
explicit factorization of $x^{\ast }$ in $E\odot F$, $\varphi
_{1}(x)=\varphi (x^{\ast })^{2}$ and $\psi _{1}(x)=\psi (x^{\ast })^{2}$.

We need only to prove that $x^{\ast }(\infty )=0$ if $I=(0,\infty )$.
Suppose $x^{\ast }(\infty )=a>0$, then, by equality (\ref{rozklad-x*}), $%
\varphi _{1}(x)^{1/2}\psi _{1}(x)^{1/2}\geq a$ for almost all $t>0$ and
considering the sets $A=\{t>0\colon \varphi (x^{\ast })\geq \sqrt{a}%
\},~B=\{t>0\colon \psi (x^{\ast })\geq \sqrt{a}\}$ we obtain $A\cup
B=(0,\infty )$ up to a set of measure zero. Then 
\begin{equation*}
H_{r}^{\ast }(\varphi _{1}(x))(t)=\Big(\int_{t}^{\infty }\frac{\varphi
_{1}(x)(s)^{r}}{s}\,ds\Big)^{1/r}\geq \Big(\int_{A\cap (t,\infty )}\frac{%
a^{r/2}}{s}\,ds\Big)^{1/r}
\end{equation*}%
and 
\begin{equation*}
H_{r}^{\ast }(\psi _{1}(x))(t)=\Big(\int_{t}^{\infty }\frac{\psi
_{1}(x)(s)^{r}}{s}\,ds\Big)^{1/r}\geq \Big(\int_{B\cap (t,\infty )}\frac{%
a^{r/2}}{s}\,ds\Big)^{1/r},
\end{equation*}%
which means $H_{r}^{\ast }(\varphi _{1}(x))(t)+H_{r}^{\ast }(\psi
_{1}(x))(t)=+\infty $ for all $t>0$. Since 
\begin{equation*}
(0,\infty )=\{t>0\colon H_{r}^{\ast }(\varphi _{1}(x))(t)=\infty \}\cup
\{t>0\colon H_{r}^{\ast }(\psi _{1}(x))(t)=\infty \}
\end{equation*}%
(maybe except a set of measure zero) it follows that $H_{r}^{\ast }(\varphi
_{1}(x))\notin E^{(1/2)}$ or $H_{r}^{\ast }(\psi _{1}(x))\notin F^{(1/2)}$,
which is a contradiction because $\varphi _{1}(x)\in E^{(1/2)}$ and $\psi
_{1}(x)\in F^{(1/2)}$. \end{proof}

%%%%%%%%%%%%%%%%%%%%%%%%%%%%%%%%%%%%% Chapter 6

\section{On factorization of symmetrized spaces}

The classical factorization theorem of Lozanovski{\u{\i}} states that for
any Banach ideal space $E$ the space $L^{1}$ has a factorization $%
L^{1}=E\odot E^{\prime }$. The natural generalization has been investigated
in \cite{KLM14} (see also \cite{Ni85}, \cite{Sc1}): for Banach ideal spaces $%
E$ and $F$, when it is possible to factorize $F$ through $E$, i.e., when the
equality 
\begin{equation*}
F=E\odot M(E,F)\,\text{holds }?
\end{equation*}%
Of course, such a natural generalization is not true without additional
assumptions on the spaces, as we can see on examples presented in \cite[%
Section 6]{KLM14}. In particular: $L^{2,1}\odot M(L^{2,1},L^{2})\equiv
L^{2,1}\odot L^{\infty }\equiv L^{2,1}\not\hookrightarrow L^{2}$ or $%
L^{2}\odot M(L^{2},L^{2,\infty })\equiv L^{2}\odot L^{\infty }\equiv
L^{2}\not\hookrightarrow L^{2,\infty }$ (see \cite[Example 2]{KLM14}).

It is easy to see that if for Banach ideal spaces $E$ and $F$ the space $F$
has a factorization through $E$, i.e., $F=E\odot M(E,F)$, then the
corresponding weighted factorization holds, that is 
\begin{equation}
F(w_{1})=E(w_{0})\odot M(E(w_{0}),F(w_{1})).  \label{weight-fac}
\end{equation}%
In fact, applying Example \ref{Ex5}(b) and property (x) from \cite{KLM13} we
get 
\begin{eqnarray*}
E(w_{0})\odot M(E(w_{0}),F(w_{1})) &=&E(w_{0})\odot M(E,F)(w_{1}/w_{0}) \\
&=&(E\odot M(E,F))(w_{1})=F(w_{1}).
\end{eqnarray*}

Applying Theorem \ref{multip-komut} and Corollary \ref{ilocz-gw-komut} we
get a result on factorization of symmetrizations.

%%%%%%%%%%%%%%%%%%%%%% Corollary 4

\begin{corollary}
\label{E-E*} Suppose that assumptions from the Theorem \ref{multip-komut}
for the spaces $E,F$ and the assumptions from Corollary \ref{ilocz-gw-komut}
for the spaces $E,M\left( E,F\right) $ are satisfied. If $F$ factorizes
through $E$, i.e., $F=E\odot M(E,F)$, then the symmetrization $F^{(\ast )}$
factorizes through the symmetrization $E^{(\ast )}$, that is, 
\begin{equation}
F^{(\ast )}=E^{(\ast )}\odot M(E^{(\ast )},F^{(\ast )}).  \label{facsym}
\end{equation}
\end{corollary}

%%%%%%%%%%%%%%%%%%%%%
\begin{proof}
We have
\begin{eqnarray*}
E^{(*)} \odot M(E^{(*)}, F^{(*)}) 
&=& 
E^{(*)} \odot M(E, F)^{(*)} \hspace{1cm}  [{\rm by ~ Theorem ~ \ref{multip-komut}}] \\
&=& 
[E \odot M(E, F)]^{(*)} \hspace{1.2cm} [{\rm by ~ Corollary \ref{ilocz-gw-komut}}] \\
&=&
F^{(*)} \hspace{3.6cm} [{\rm by ~ assumption}].  
\end{eqnarray*} 
\end{proof}

Consequently, from (\ref{weight-fac}) and Corollary \ref{E-E*} we can get
that, under some assumptions on Banach ideal spaces $E, F$, if $F = E \odot
M(E, F)$, then 
\begin{equation*}
E(w_{0})^{(*)} \odot M(E(w_{0})^{(*)},F(w_{1})^{(*)}) = F(w_{1})^{(*)}.
\end{equation*}

Then the factorization of classical spaces of the type $E^{(*)}$ like
Lorentz and Marcinkiewicz spaces (see \cite[Examples 3 and 4]{KLM14}) comes
also from the known factorization of a respective weighted $L^{p}$-spaces or
just $L^{p}$ spaces.

%%%%%%%%%%%%%%%%%%%%%%%%%%%%%%%%%%%%% Chapter 7

\section{On multipliers and factorization of Ces\`aro function spaces}

We will need result on the cancellation property for the product spaces and
the factorization property of the K\"{o}the dual.

%%%%%%%%%%%%%%%%%%%%%%% Lemma 2

\begin{lemma}
\label{Lemma2} Let $E, F, G$ be quasi-Banach ideal spaces with the Fatou
property. Assume that for some $p > 0$ all three $p$-convexifications $%
E^{(p)}, F^{(p)}, G^{(p)}$ are normable.

\begin{itemize}
\item[$(i)$] If $E \odot F = E \odot G$ then $F = G$.

\item[$(ii)$] Suppose additionally that $E,F$ are Banach spaces. If $F$
factorizes through $E$, i.e., $F=E\odot M(E,F)$, then the K\"{o}the dual $%
E^{\prime }$ factorizes through the K\"{o}the dual $F^{\prime }$, that is, $%
E^{\prime }=F^{\prime }\odot M(F^{\prime },E^{\prime })$.
\end{itemize}
\end{lemma}

%%%%%%%%%%%%%%%
%TCIMACRO{\TeXButton{proof}{\proof}}%
%BeginExpansion
\proof%
%EndExpansion
(i) Since, by \cite[Theorem 1 (iii)]{KLM14}, $E\odot F=(E^{(p)}\odot
F^{(p)})^{(1/p)}$ and similarly for $E\odot G$ we conclude by the assumption
that 
\begin{equation*}
(E^{(p)}\odot F^{(p)})^{(1/p)}=(E^{(p)}\odot G^{(p)})^{(1/p)}~~\mathrm{and~so%
}~~E^{(p)}\odot F^{(p)}=E^{(p)}\odot G^{(p)}.
\end{equation*}%
For Banach ideal spaces $E,F$, however, we have by \cite[Theorem 1 (iv)]%
{KLM14} that the product space is $1/2$-convexification of the Calder\'{o}n
product: $E\odot F=(E^{1/2}F^{1/2})^{(1/2)}$ and we obtain 
\begin{equation*}
\left[ (E^{(p)})^{1/2}(F^{(p)})^{1/2}\right] ^{(1/2)}=\left[
(E^{(p)})^{1/2}(G^{(p)})^{1/2}\right] ^{(1/2)},
\end{equation*}%
which gives $(E^{(p)})^{1/2}(F^{(p)})^{1/2}=(E^{(p)})^{1/2}(G^{(p)})^{1/2}$
and by uniqueness of Calder\'{o}n--Lo\-za\-nov\-ski{\u{\i}} construction 
\cite[Corollary 1]{BM05} (see also \cite[Theorem 3.5]{CNS03} with a direct
proof) we get $F^{(p)}=G^{(p)}$ or $F=G$.

(ii) If $F=E\odot M(E,F)$, then by the Lozanovski{\u{\i}} factorization
theorem 
\begin{equation}
F^{\prime }\odot M(E,F)\odot E=F^{\prime }\odot F=L^{1}=E^{\prime }\odot E.
\label{Lem2}
\end{equation}%
Observe that, by property (xi) in \cite{KLM13} the space $M(E,F)$ has the
Fatou property, and by Corollary 1 (ii) in \cite{KLM14} the space $F^{\prime
}\odot M(E,F)$ has the Fatou property. Thus, by (i) above, we get $E^{\prime
}=F^{\prime }\odot M(E,F)$ and finally, by property (vii) in \cite{KLM13}
saying that $M(E,F)=M(F^{\prime },E^{\prime })$ we get $E^{\prime
}=F^{\prime }\odot M(F^{\prime },E^{\prime })$.%
%TCIMACRO{\TeXButton{proof}{\endproof}}%
%BeginExpansion
\endproof%
%EndExpansion

Now we disscuss the factorization of Ces\`{a}ro spaces. Recall that for a
Banach ideal space $E$ on $I$ the \textit{Ces\`{a}ro function space} $%
CE=CE(I)$ is defined as 
\begin{equation}
CE=\{f\in L^{0}(I)\colon H\left\vert f\right\vert \in E\}~~\mathrm{%
with~the~norm}~~\Vert f\Vert _{CE}=\Vert H\left\vert f\right\vert \Vert _{E},
\end{equation}%
and the \textit{Tandori function space} $\widetilde{E}=\widetilde{E}(I)$ as 
\begin{equation}
\widetilde{E}=\{f\in L^{0}(I):\widetilde{f}\in E\}~~\mathrm{with~the~norm}%
~~\Vert f\Vert _{\widetilde{E}}=\Vert \widetilde{f}\Vert _{E},  \label{falka}
\end{equation}%
where $H$ is a Hardy operator and $\widetilde{f}(x)=\mathrm{ess}\sup_{t\in
I,\,t\geq x}|f(t)|$ (cf. \cite{AM09}, \cite{LM15}, \cite{LM16}). For
example, if $E=L^{p}(I)$ the respective space $CL^{p}(I)$ is the classical
Ces\`{a}ro function space denoted usually by $Ces_{p}\left( I\right) .$
Similarly, in the sequence case $E=l^{p}$ we have $ces_{p}:=Cl^{p}.$

%%%%%%%%%%%%%%%%%%  Theorem 6

\begin{theorem}
\label{mult-cesaro} Let $E,F$ be symmetric Banach function spaces on $%
I=(0,\infty )$ (or symmetric Banach sequence spaces) with the Fatou property
such that the operator $H$ is bounded on $E$ and on $F$. Assume that $F$
factorizes through $E$, that is, $F=E\odot M(E,F)$. Then 
\begin{equation}
M(CE,CF)=\widetilde{M(E,F)}.  \label{Thm6}
\end{equation}
\end{theorem}

%%%%%%%%%%%%%%%%%%%

\begin{proof} Let $E,F$ be symmetric Banach function spaces on $I=(0,\infty )
$. Note that $CE\neq \{0\},CF\neq \{0\}$ and $(CE)^{\prime }=\widetilde{%
E^{\prime }},(CF)^{\prime }=\widetilde{F^{\prime }}$ (see \cite[Theorem 1 and 2]%
{LM15}). Denote $G=M(E,F)$. The space of multipliers $G$ is a symmetric space
(see \cite[Theorem 2.2(i)]{KLM13}) and $G=M(F^{\prime },E^{\prime })$ (cf. \cite[%
property (vii)]{KLM13}).

First, we show that 
\begin{equation*}
\widetilde{F^{\prime }}\odot \widetilde{G}=\widetilde{F^{\prime }\odot G}.
\end{equation*}%
In fact, applying Theorem 1(iv) from \cite{KLM14}, Theorem 4 from \cite{LM16}
(since $F^{\prime }$ and $G$ are symmetric), the equality $\widetilde{E^{(p)}}=(%
\widetilde{E})^{(p)}$ and again Theorem 1(iv) from \cite{KLM14} we obtain 
\begin{equation*}
\widetilde{F^{\prime }}\odot \widetilde{G}=[(\widetilde{F^{\prime }})^{1/2}%
\widetilde{G}^{1/2}]^{(1/2)}=[\widetilde{{(F^{\prime })}^{1/2}G^{1/2}}%
]^{(1/2)}=\widetilde{F^{\prime }\odot G}.
\end{equation*}%
Second, since $F=E\odot M(E,F)$ it follows by the Lozanovski{\u{\i}}
factorization theorem that 
\begin{equation*}
\lbrack F^{\prime }\odot M(E,F)]\odot E=F^{\prime }\odot \lbrack M(E,F)\odot
E]=F^{\prime }\odot F=L^{1}=E^{\prime }\odot E,
\end{equation*}%
and by Lemma \ref{Lemma2}(i) we get $E^{\prime }=F^{\prime }\odot M(E,F)$.
Thus, 
\begin{equation*}
\widetilde{F^{\prime }}\odot \widetilde{G}=\widetilde{F^{\prime }\odot G}=%
\widetilde{E^{\prime }}.
\end{equation*}%
Using the last equality, Theorem 2 from \cite{LM15} on the K\"{o}the duality
of abstract Ces\`{a}ro spaces $(CE)^{\prime }=\widetilde{(E^{\prime })}$ and
the Lozanovski{\u{\i}} factorization theorem we obtain 
\begin{equation}
CE\odot \widetilde{F^{\prime }}\odot \widetilde{G}=CE\odot \widetilde{%
E^{\prime }}=CE\odot (CE)^{\prime }=L^{1}.  \label{Thm6a}
\end{equation}%
Taking $L^{1}=(\widetilde{G})^{\prime }\odot \widetilde{G}$ in (\ref{Thm6a})
and applying Lemma \ref{Lemma2}(i), we obtain $CE\odot \widetilde{F^{\prime }%
}=(\widetilde{G})^{\prime }$, whence 
\begin{equation}
(CE\odot \widetilde{F^{\prime }})^{\prime }=(\widetilde{G})^{\prime \prime }=%
\widetilde{G}=\widetilde{M(E,F)}.  \label{Thm6b}
\end{equation}%
Applying Theorem 4 from \cite{KLM14}, Theorem 2 from \cite{LM15}, the
Lozanovski{\u{\i}} factorization theorem, the K\"{o}the duality $%
(CF)^{\prime }=\widetilde{(F^{\prime })}$ and the identification (\ref{Thm6b}%
) we obtain 
\begin{eqnarray*}
M(CE,CF) &=&M[CE\odot (CF)^{\prime },CF\odot (CF)^{\prime }]=M(CE\odot
(CF)^{\prime },L^{1}) \\
&=&[CE\odot (CF)^{\prime }]^{\prime }=(CE\odot \widetilde{F^{\prime }}%
)^{\prime }=\widetilde{M(E,F)}.
\end{eqnarray*}

The proof is the same for symmetric Banach sequence spaces, applying Theorem
6 instead of Theorem 2 from \cite{LM15}. \end{proof}

\begin{remark}
Note that the above theorem for Banach function spaces on $I=(0,\infty )$ is
also true with some different set of assumptions. Namely, if $E,F$ are
Banach ideal spaces on $I=(0,\infty )$ with the Fatou property such that
both the operators $H,H^{\ast }$ and $D_{\tau }$ are bounded on $E$ and on $%
F,$ for some $\tau \in \left( 0,1\right) ,$ then it is enough to apply in
the proof Theorem 3 instead of Theorem 2 from {\rm{\cite{LM15}}}. 
%%%%%%%%%%%%%%%%% Example 7
\end{remark}

\begin{example}
\label{Ces_p}\emph{{{\ Let $1<q\leq p\leq \infty $. Set $\frac{1}{r}=\frac{1%
}{q}-\frac{1}{p}.$ Then 
\begin{equation}
M(Ces_{p}\left( I\right) ,Ces_{q}\left( I\right) )=M(\widetilde{L^{p}\left(
I\right) },\widetilde{L^{q}\left( I\right) })=\widetilde{L^{r}\left(
I\right) }\text{ with }I=(0,\infty )  \label{Cesp}
\end{equation}%
and 
\begin{equation*}
M(ces_{p},ces_{q})=\widetilde{l^{r}},
\end{equation*}%
where $\widetilde{l^{r}}=\{x=(x_{n})\colon (\sum\limits_{n=1}^{\infty
}\sup\limits_{k\geq n}|x_{k}|^{r})^{1/r}<\infty \}.$ } } }
\end{example}

%%%%%%%%%%%%%%%%
\begin{proof}
Since for $1< q\leq p\leq \infty $ we have $M(L^{p}(I),L^{q}(I))\equiv
L^{r}(I)$ (cf. \cite[Proposition 3]{MP89}) and $L^{p}(I)\odot
L^{r}(I)=L^{q}(I)$ (cf. \cite[p. 1373]{CS17} and \cite[Example 1(a)]{KLM14}%
), where $\frac{1}{r}=\frac{1}{q}-\frac{1}{p}$, then using Theorem \ref%
{mult-cesaro} with necessary restrictions on $p,q$ we obtain 
\begin{equation*}
M(Ces_{p}(I),Ces_{q}(I))=\widetilde{M[L^{p}(I),L^{q}(I)]}=\widetilde{L^{r}(I)%
}.
\end{equation*}%
Also for $I=(0,\infty )$ we have 
\begin{equation*}
M(\widetilde{L^{p}},\widetilde{L^{q}})=M[(\widetilde{L^{q}})^{\prime },(%
\widetilde{L^{p}})^{\prime }]=M(Ces_{q^{\prime }},Ces_{p^{\prime }})=%
\widetilde{L^{r}},
\end{equation*}%
since $1/p^{\prime }-1/q^{\prime }=1/q-1/p=1/r$. For the sequence case the
proof is the same.
\end{proof}

Note that C. Bennett proved the above result $M(ces_{p},ces_{q})=\widetilde{%
l^{r}}$ for Ces\`{a}ro sequence spaces in \cite{Be96}.%\mathbb%

\begin{problem}
Prove an analogous result to Theorem \ref{mult-cesaro} for $I=\left(
0,1\right) .$
\end{problem}

We need to assume that $E,F$ are symmetric Banach function spaces with the
Fatou property such that the operators $H$ and $H^{\ast }$ are bounded on $E$
and on $F$. Then, by Corollary 13 from \cite{LM15} on the K\"{o}the duality
of abstract Ces\`{a}ro spaces on $I=(0,1),$ we have $(CE)^{\prime }=%
\widetilde{(E^{\prime }(1/w))}$ for $w(t)=1-t,t\in I$. Suppose we try to
prove this result similarly as for $I=\left( 0,\infty \right) .$
Unfortunately, we are not able to apply Theorem 4 from \cite{LM16} because
the respective space $F^{\prime }(1/w)$ is not symmetric. Thus, for Theorem
4 from \cite{LM16}, we need to assume that $H,H^{\ast }$ are bounded on $%
F^{\prime }(1/w)$ and $M\left( E,F\right) $ which do not seem to be
reasonable.

\bigskip

Note that we can not apply Theorem \ref{mult-cesaro} in the case when $%
M(E,F)=L^{\infty }$ with $E\neq F$ or $M(E,F)=\left\{ 0\right\} ,$ because
the factorization assumption is not satisfied. However, for $1<p\leq
q<\infty $ we have $M(l^{p},l^{q})=l^{\infty }$ and $M(L^{p}\left( I\right)
,L^{q}\left( I\right) )=\left\{ 0\right\} .$ Consequently, it is natural to
find descriptions of\ $M(CE,CF)$ in this cases. Note that C. Bennett \cite%
{Be96} proved that if $1<p\leq q<\infty $, then $M(ces_{p},ces_{q})=%
\{x=(x_{n})\colon \sup\limits_{n\in {N}}n^{1/q-1/p}|x_{n}|<\infty \}$.

%%%%%%%%%%%%%%%%%%%% Corollary 3

\begin{corollary}
\label{faktor-Cesaro}Suppose the assumptions of Theorem \ref{mult-cesaro}
are satisfied. Then Ces\`{a}ro function (sequence) space $CF$ can be
factorized by another Ces\`{a}ro function (sequence) space $CE$, that is, 
\begin{equation*}
CF=CE\odot M(CE,CF).
\end{equation*}
\end{corollary}

%%%%%%%%%%%%%%%%%%%%% 
\begin{proof}
(i) Let $E,F$ be symmetric Banach function spaces on $I = (0, \infty )$. Applying equality (\ref{Thm6a}) we get 
$CE \odot \widetilde{F^{\prime}} \odot \widetilde{M(E, F)} = L^{1}$. 
By our Theorem \ref{mult-cesaro} and Theorem 2 from \cite{LM15} we conclude that
$$
CE \odot M(CE, CF) \odot \widetilde{F^{\prime}} = L^{1} = CF \odot (CF)^{\prime} = CF \odot \widetilde{F^{\prime}},
$$
whence, by Lemma \ref{Lemma2}(i),
$$
CE \odot M(CE, CF) = CF.
$$
The proof in the sequence case is the same.
\end{proof}

{Since factorization of Lebesgue spaces (Example \ref{Ces_p}) and Orlicz
spaces (Theorem 2 in \cite{LT17}, cf. also Theorem 9 and Corollary 8 in \cite%
{KLM14}) is known, it is easy to conclude the respective factorization of Ces%
\`{a}ro function spaces $Ces_{p}$, Ces\`{a}ro--Orlicz function spaces $%
Ces_{\varphi }.$ Note that we may consider also different weighted Ces\`{a}%
ro function spaces} 
\begin{equation*}
Ces_{p}\left( w\right) =CL^{p}\left( w\right) =\left\{ x\in L^{0}:xw\in
CL^{p}\right\} =\left\{ x\in L^{0}:H\left\vert xw\right\vert \in
L^{p}\right\}
\end{equation*}%
or%
\begin{equation*}
C\left( L^{p}\left( w\right) \right) =\left\{ x\in L^{0}:H\left\vert
x\right\vert \in L^{p}\left( w\right) \right\} =\left\{ x\in
L^{0}:wH\left\vert x\right\vert \in L^{p}\right\} .
\end{equation*}

Then applying our results one can conclude the respective factorization of
spaces $Ces_{p}\left( w\right) $ and $C\left( L^{p}\left( t^{\alpha }\right)
\right) .$ %%%%%%%%%%%%%%%%%%%%%%%%%%%%%%%%%%%%%%%%%%%  

\section{Acknowledgements}

The author Pawe\l\ Kolwicz and Karol Le\'{s}nik are supported by the Ministry of Science and Higher Education
of Poland, grant number 04/43/DSPB/0094.

\vspace{3mm}

\noindent {\footnotesize Pawe\l\ Kolwicz and Karol Le\'{s}nik, Institute of
Mathematics, Faculty of Electrical Engineering,\newline
Pozna\'{n} University of Technology, ul. Piotrowo 3a, 60-965 Pozna\'{n},
Poland }\newline
\textit{E-mails:} ~\texttt{pawel.kolwicz@put.poznan.pl, klesnik@vp.pl}%
\newline

\vspace{-1mm}

\noindent {\footnotesize Lech Maligranda, Department of Engineering Sciences
and Mathematics\newline
Lule\aa\ University of Technology, SE-971 87 Lule\aa , Sweden}\newline
~\textit{E-mail:} \texttt{lech.maligranda@ltu.se }

\end{document}